\pdfoutput=1
\RequirePackage{ifpdf}
\ifpdf 
\documentclass[pdftex]{sigma}
\else
\documentclass{sigma}
\fi

\usepackage{ytableau}

\numberwithin{equation}{section}

\newtheorem{Theorem}{Theorem}[section]
\newtheorem{Corollary}[Theorem]{Corollary}
\newtheorem{Conjecture}[Theorem]{Conjecture}
\newtheorem{Lemma}[Theorem]{Lemma}
\newtheorem{Proposition}[Theorem]{Proposition}
{ \theoremstyle{definition}
 \newtheorem{Definition}[Theorem]{Definition}
 \newtheorem{Example}[Theorem]{Example}
 }

\DeclareMathOperator{\Wr}{Wr}
\DeclareMathOperator{\He}{He}
\DeclareMathOperator{\sgn}{sgn}
\DeclareMathOperator{\Hor}{Hor}
\DeclareMathOperator{\Ver}{Ver}

\begin{document}

\newcommand{\arXivNumber}{1801.07980}

\renewcommand{\thefootnote}{}

\renewcommand{\PaperNumber}{048}

\FirstPageHeading

\ShortArticleName{Recurrence Relations for Wronskian Hermite Polynomials}

\ArticleName{Recurrence Relations\\ for Wronskian Hermite Polynomials\footnote{This paper is a~contribution to the Special Issue on Orthogonal Polynomials, Special Functions and Applications (OPSFA14). The full collection is available at \href{https://www.emis.de/journals/SIGMA/OPSFA2017.html}{https://www.emis.de/journals/SIGMA/OPSFA2017.html}}}

\Author{Niels BONNEUX and Marco STEVENS}
\AuthorNameForHeading{N.~Bonneux and M.~Stevens}
\Address{Department of Mathematics, University of Leuven,\\ Celestijnenlaan 200B box 2400, 3001 Leuven, Belgium}
\Email{\href{mailto:niels.bonneux@kuleuven.be}{niels.bonneux@kuleuven.be}, \href{mailto:marco.stevens@kuleuven.be}{marco.stevens@kuleuven.be}}

\ArticleDates{Received January 25, 2018, in final form May 09, 2018; Published online May 16, 2018}

\Abstract{We consider polynomials that are defined as Wronskians of certain sets of Hermite polynomials. Our main result is a recurrence relation for these polynomials in terms of those of one or two degrees smaller, which generalizes the well-known three term recurrence relation for Hermite polynomials. The polynomials are defined using partitions of natural numbers, and the coefficients in the recurrence relation can be expressed in terms of the number of standard Young tableaux of these partitions. Using the recurrence relation, we provide another recurrence relation and show that the average of the considered poly\-no\-mials with respect to the Plancherel measure is very simple. Furthermore, we show that some existing results in the literature are easy corollaries of the recurrence relation.}


\Classification{05A17; 12E10; 26C05; 33C45; 65Q30}

\renewcommand{\thefootnote}{\arabic{footnote}}
\setcounter{footnote}{0}

\section{Introduction}\label{sectionintroduction}

Over the last decade, the theory of exceptional orthogonal polynomials emerged as a new research domain in classical analysis \cite{GomezUllate_Kamran_Milson-09,Odake_Sasaki-09}. These polynomials are defined just like classical orthogonal polynomials, albeit with the possibility that for a finite number of degrees no polynomial exists. In the primary examples, these exceptional orthogonal polynomials appear as Wronskians of a~set of classical orthogonal polynomials \cite{Szego}, and as such are called exceptional Hermite \cite{Duran-Hermite,GomezUllate_Grandati_Milson-14}, exceptional Laguerre \cite{Bonneux_Kuijlaars,Duran-Laguerre} and exceptional Jacobi \cite{Duran-Jacobi} polynomials.

An important property of the classical orthogonal polynomials is that they satisfy a three term recurrence relation, hence it is interesting to see whether exceptional orthogonal polynomials also do. The main issue in this regard is the fact that for exceptional orthogonal polynomials there is not a polynomial of every degree. Several papers are published concerning new recurrence relations for exceptional polynomials but none of these give general explicit formulas for the coefficients in the relation. Most work was done for exceptional Hermite polynomials, see for example \cite{Gomez-Ullate_Kasman_Kuijlaars_Milson} and the references therein.

In this paper, we focus on the class of exceptional Hermite polynomials. In fact, we forget about the orthogonality structure, and simply consider the polynomials that are defined as the Wronskian of an appropriate set of Hermite polynomials. Those appropriate sets are labeled by partitions of natural numbers and we refer to the resulting polynomials as \emph{Wronskian Hermite polynomials}. These polynomials also appear as building blocks in rational solutions of the Painlev\'e~IV equation \cite{Clarkson-3,Clarkson-1,Clarkson-2,Noumi_Yamada,Okamoto}.

We establish a generating recurrence relation with explicit coefficients for these Wronskian Hermite polynomials, see Theorem \ref{thm:RecurrenceRelation}. These coefficients are well-known objects in the combinatorics of partitions, namely the number of standard Young tableaux \cite{Young-1900,Young-1928}. Hence, this work leads us to believe that the use of partitions in the description of exceptional orthogonal polynomials is indeed the natural setting. Furthermore, when interpreting the Hermite polynomials in the Wronskian setting, we recover the original recurrence relation for Hermite polynomials.

Besides a recurrence relation, we also present two other main results concerning Wronskian Hermite polynomials. Firstly, there is another recurrence relation as described in Theorem~\ref{thm:topdownrecurrence}, which does not generate all the polynomials, but which does reveal some of the structure of these polynomials. Secondly, we present a result concerning taking averages of Wronskian Hermite polynomials. Namely, fix a degree and consider all partitions of this degree. Then the average with respect to the Plancherel measure \cite{Baik_Deift_Suidan} of the Wronskian Hermite polynomials associated to these partitions, is precisely the monomial of the degree we started with, see Theorem \ref{thm:Average}.

We also show that the generating recurrence relation can be used to prove some of the results in the existing literature in a very easy manner, whereas they used to have rather involved proofs. Precisely for this reason, we believe that the recurrence relation we provide will help to derive new results concerning Wronskian Hermite polynomials. For example, it is conjectured that all the non-zero roots of these polynomials are simple \cite{Felder_Hemery_Veselov,GarciaFerrero_GomezUllate,Kuijlaars_Milson}; perhaps this recurrence relation can inspire progress in this regard. Moreover, our recurrence relation provides a way to inductively prove explicit formulas for the coefficients of the polynomials.

In the existing literature, Wronskian Hermite polynomials are defined using classical Hermite polynomials~\cite{Szego}. In this paper, we use a slightly modified version of Hermite polynomials (that are sometimes referred to as \emph{probabilists' Hermite polynomials}~\cite{O'Donnell}) instead of the classical polynomials. The reason for doing this lies in the fact that when using the classical Hermite polynomials, there is not one natural choice of normalization for which all of our results hold. However, all of our results do have an interpretation in terms of Wronskians of classical Hermite polynomials. See Section~\ref{classicalHermite} for a detailed discussion.

\looseness=-1 This paper is constructed as follows. We start with some preliminaries about (Wronskian) Hermite polynomials, partitions and related concepts. In Section~\ref{sectionmainresults} we state our results: the two recurrence relations for Wronskian Hermite polynomials (Theorem~\ref{thm:RecurrenceRelation} and Theorem~\ref{thm:topdownrecurrence}) and the average with respect to the Plancherel measure (Theorem~\ref{thm:Average}). Moreover we add known results for Wronskian Hermite polynomials which follow straightforwardly from the recurrence relation. Section~\ref{sectionnotesandremarks} contains a discussion on the use of classical Hermite polynomials and we discuss some notes and remarks about the recurrence relation. We also state some special cases that are connected to primary examples in the existing literature. The proof of the generating recurrence relation is given in Section~\ref{sectionproofrecursion}. Subsequently, Section~\ref{sectiontopdownrecurrence} contains the proof of the second recurrence relation. In Section~\ref{sectionproofaverage} we prove the result regarding the average polynomial. In Section~\ref{sectionproofderivative}, we give a proof of an explicit expression for the derivative of Wronskian Hermite polynomials, which is one of the corollaries of the recurrence relation, see Proposition~\ref{prop:OmegaDer}. Finally, we added an appendix that contains proofs of technical lemmas that are needed in the proofs of our main results.

\section{Preliminaries}\label{sectionpreliminaries}
In this section we present all the necessary concepts and notation for the results in the next sections.

\subsection{Hermite polynomials}\label{subsectionstandardhermitepolynomials}
One of the most studied sets of polynomials is the set of Hermite polynomials. They arise in a natural way in many different fields; to name a few of the primary examples, they appear as the polynomial part of the eigenfunctions of the (quantum) harmonic oscillator but also in the statistics of normal distributions. Here, we use the polynomials denoted by~$\He_n$. These polynomials are sometimes referred as the \emph{probabilists'} Hermite polynomials~\cite{O'Donnell}, due to their use in the study of the normal distribution. Many published works are devoted to the study of Hermite polynomials, but for our purpose we introduce them by means of their recurrence relation \cite[formula~(22.7.14)]{Abramowitz_Stegun}.

We define $\He_0(x)=1$ and $\He_1(x)=x$. Furthermore, we recursively define
\begin{gather}\label{eq:recursionforstandardhermite}
\He_n(x)=x\He_{n-1}(x)-(n-1)\He_{n-2}(x), \qquad n\geq 2.
\end{gather}
These polynomials satisfy the orthogonality relations \cite[formula~(22.2.15)]{Abramowitz_Stegun}
\begin{gather*}
\frac{1}{\sqrt{2\pi}} \int_{-\infty}^{\infty} \He_n(x)\He_m(x) e^{-\frac{x^2}{2}}{\rm d}x= n! \delta_{n,m},
\end{gather*}
where $\delta_{n,m}$ is the Kronecker delta symbol. Indeed, these polynomials are orthogonal with respect to the standard normal distribution on the real line.

We note that this definition of Hermite polynomials is not the most used form in classical analysis \cite{Szego}; the \emph{classical Hermite polynomials} are often denoted by $H_n$. For more details, see Section~\ref{classicalHermite} where we translate the upcoming results in terms of classical Hermite polynomials.

\subsection{Partitions and standard Young tableaux}\label{subsectionpartitions}
A \textit{partition} $\lambda$ of a natural number $n$, denoted by $\lambda \vdash n$, is a sequence of natural numbers $\left(\lambda_i\right)_{i=1}^r$ such that $\lambda_1\geq \lambda_2 \geq \dots \geq \lambda_r>0$ and $\sum\limits_{i=1}^r \lambda_i=n$. We denote $\lvert \lambda \rvert=n$ and call $r$ the \emph{length} of the partition $\lambda$. A standard reference for the theory of partitions is the book of Andrews \cite{Andrews}.

Partitions can be visualized by \emph{Young diagrams}: for a partition $\lambda=(\lambda_1,\dots,\lambda_r)$, we draw a~diagram of $r$ rows, where the $i$'th row consists of $\lambda_i$ boxes \cite{Young-1900,Young-1928}. For example, for the partition $\lambda=(4,2,1)$, the Young diagram looks like:
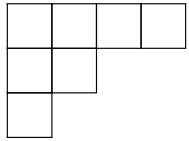
\begin{center}
 \ydiagram{4,2,1}
\end{center}

For a partition $\lambda$, the \emph{conjugate partition} $\lambda'$ is defined by
\begin{gather*}
\lambda'_j=\# \{i \,|\, \lambda_i\geq j\}, \qquad j=1,\dots,\lambda_1.
\end{gather*}
On the level of Young diagrams, this means that one reflects the Young diagram in the diagonal, as shown in the following diagram.
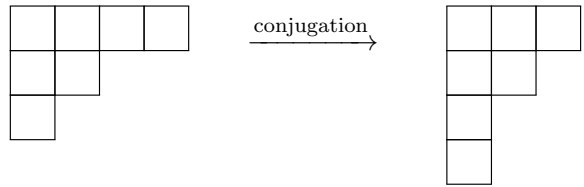
\begin{center}
 \ydiagram{4,2,1}
 $\qquad \xrightarrow{\text{conjugation}} \qquad$
 \ydiagram{3,2,1,1}
\end{center}

Building further upon this visualization, if $\lambda$ is a partition of $n$, we can fill the boxes of the associated Young diagram of \emph{shape} $\lambda$ by the integers $1$ to $n$ in such a way that the numbers in every row and column are increasing. For example, for the Young diagram above of shape $(4,2,1)$, we can have:
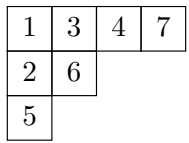
\begin{center}
 \begin{ytableau}
 1 & 3 & 4 & 7\\
 2 & 6 \\
 5
 \end{ytableau}
\end{center}
Every possible way of filling the Young diagram under these constraints is called a \emph{standard Young tableau}. The number of standard Young tableaux of shape $\lambda$ is denoted by $F_\lambda$ \cite[Chapter~3]{Baik_Deift_Suidan}. By convention, we set $F_{\varnothing}=1$ where $\varnothing$ refers to the partition of 0.

These numbers $F_\lambda$ have an interpretation in the setting of the partial ordering on all partitions. This ordering is (canonically) defined as
\begin{gather*}\mu \leq \lambda \iff \mu_i \leq \lambda_i \quad \textrm{for all} \ i,\end{gather*}
where we set $\lambda_i=0$ if $i$ is bigger than the length of $\lambda$. Note that on the level of Young diagrams, $\mu\leq \lambda$ precisely means that $\mu$ is a subdiagram of $\lambda$.

There are explicit formulas available for $F_\lambda$. For this, we first need to define the \emph{degree sequence} associated to a partition. Namely, if we let $\lambda=(\lambda_1,\dots,\lambda_r)$ be a partition, then
\begin{gather}\label{eq:DegreeSequence}
n_\lambda:=(\lambda_1+r-1,\lambda_2+r-2,\dots,\lambda_r)
\end{gather}
is called the \textit{degree sequence} of $\lambda$. Indeed, we have $(n_{\lambda})_i=\lambda_i+r-i$ for $i=1,\dots,r$. Therefore the degree sequence is strictly decreasing (whereas the partition is weakly decreasing). If there is no ambiguity which partition is associated to the degree sequence, we simply write~$n_i$ instead of~$(n_\lambda)_i$. Subsequently, we denote the \textit{Vandermonde determinant} of the degree sequence of~$\lambda$ by~$\Delta(n_{\lambda})$, i.e.,
\begin{gather*}
\Delta(n_{\lambda}) = \prod_{1\leq i < j \leq r} (n_j-n_i).
\end{gather*}
Using this notation, the explicit formula for $F_\lambda$ that is used most often in this article is
\begin{gather}\label{eq:Flambda}
F_{\lambda} = (-1)^{\frac{r(r-1)}{2}} \frac{|\lambda|! \Delta(n_{\lambda})}{\prod\limits_{i=1}^{r}n_i!}.
\end{gather}
Note that the sign of the Vandermonde determinant is precisely given by $(-1)^{\frac{r(r-1)}{2}}$ and therefore the right hand side of~\eqref{eq:Flambda} is positive. There are also two other well-known formulas for $F_\lambda$, i.e., a~determinantal formula and a hook formula. Proofs and comments can be found in \cite[Section~3.2]{Baik_Deift_Suidan}.

The numbers $F_\lambda$ have a further remarkable property. Namely, for any fixed natural num\-ber~$n$, we have
\begin{gather*}
\sum_{\lambda \vdash n} F_\lambda^2 = n!,
\end{gather*}
see \cite[formulas (1.30) and~(3.17)]{Baik_Deift_Suidan}. Therefore, we find a probability distribution on the set $\{\lambda \colon \lambda \vdash n\}$ where every partition $\lambda$ has weight $\frac{F_\lambda^2}{n!}$. This probability distribution is called the \textit{Plancherel measure} \cite{Baik_Deift_Suidan}.

\subsection{Wronskian Hermite polynomials}\label{subsectiongeneralizedhermitepolynomials}
For a set of (sufficiently differentiable) functions $\{f_i\}_{i=1}^r$, the \textit{Wronskian} $\Wr[f_1,\dots,f_r]$ is defined as the determinant of the $r\times r$ matrix $M$, where
\begin{gather*}
M_{ij}=f_j^{(i-1)}, \qquad 1\leq i,j\leq r.
\end{gather*}
In the setting where all the functions $\{f_i\}_{i=1}^r$ are polynomials, we have the following well-known result in terms of partitions. This can be proven directly by evaluating the determinant.

\begin{Lemma}\label{lem:WronskianPolynomial}
 Suppose that $\lambda$ is a partition with degree sequence $n_{\lambda}$ and let $P_1,\dots,P_r$ be monic polynomials such that $\deg P_i = n_i$ for $i=1,\dots,r$. Then the Wronskian $\Wr[P_1,\dots,P_r]$ is a~polynomial of degree $|\lambda|$ and its leading coefficient is the Vandermonde determinant $\Delta(n_{\lambda})$.
\end{Lemma}

Now we specify this to the Wronskian of Hermite polynomials. This definition is a slight modification of the one used in the existing literature \cite{Felder_Hemery_Veselov, GomezUllate_Grandati_Milson-14,GomezUllate_Grandati_Milson-16,Kuijlaars_Milson}.

\begin{Definition}\label{def:WronskianHermite}
 For any partition $\lambda$, we define the \textit{Wronskian Hermite polynomial} associated to $\lambda$ as
 \begin{gather}\label{eq:OmegaHe}
 \He_{\lambda}:=\frac{\Wr [\He_{n_1},\dots,\He_{n_r}]}{\Delta(n_\lambda)},
 \end{gather}
 where $n_{\lambda}=(n_1,\dots,n_r)$ is the degree sequence of $\lambda$, see \eqref{eq:DegreeSequence}.
\end{Definition}

By Lemma \ref{lem:WronskianPolynomial}, we have that $\He_{\lambda}$ is a monic polynomial of degree $|\lambda|$. We also note that this definition generalizes the concept of Hermite polynomials. Namely, for the partition $\lambda=(n)$, we have $\He_{(n)}\equiv\He_n$.

\subsection{New notation}
\label{subsectionnotation}
In Section~\ref{sectionmainresults} we state our main results. For this, we need some (new) notation which we introduce here.

Using the partial order on all partitions defined in Section~\ref{subsectionpartitions}, we can define the sets
\begin{gather}\label{eq:Tk}
T_k(\lambda)=\{\mu\leq \lambda \,|\, \lvert \mu \rvert=\lvert \lambda\rvert -k\}
\end{gather}
for all partitions $\lambda$ and all $k\leq \lvert \lambda \rvert$. We give special attention to some elements of $T_2(\lambda)$: we define
\begin{gather*}
T_2^h(\lambda)=\{\rho\in T_2(\lambda) \,|\, \exists\, i\colon \rho_i=\lambda_i-2\}
\end{gather*}
and
\begin{gather*}
T_2^v(\lambda)=\{\rho \in T_2(\lambda) \,|\, \exists \, i\colon \rho_i=\rho_{i+1}=\lambda_{i}-1=\lambda_{i+1}-1\}.
\end{gather*}
Again, whenever $i$ is bigger than the length of $\mu$, we naturally set $\mu_i=0$. These sets can be understood most easily by considering Young diagrams as illustrated in Example~\ref{ex:644211}. Indeed, the~$h$ in $T_2^{h}$ represents the fact that the two deleted cells in the Young diagram of $\lambda$ are horizontally adjacent. Likewise, the elements in~$T_2^{v}$ are obtained by deleting two vertically adjacent cells.
We define
\begin{gather*}
\tilde{T}_2(\lambda)=T_2^h(\lambda)\cup T_2^v(\lambda).
\end{gather*}
From the above definition, it is also clear that one has the relation
\begin{gather}\label{eq:verticalhorizontalconjugate}
\rho \in T_2^v(\lambda) \iff \rho' \in T_2^h(\lambda')
\end{gather}
for every partition $\lambda$. Finally, we define the symbol $\sgn(\rho,\lambda)$ for every partition $\rho\in \tilde{T}_2(\lambda)$ as
\begin{gather}\label{eq:sgn}
\sgn(\rho,\lambda)
= \begin{cases}
-1, & \rho\in T_2^h(\lambda), \\
\hphantom{-}1, & \rho \in T_2^v(\lambda).
\end{cases}
\end{gather}

\begin{Example}\label{ex:644211}
 Consider the partition $\lambda=(6,4,4,2,1,1)$. In Fig.~\ref{fig:1}, the crosses and bullets indicate which cells can be removed to obtain elements in $T_1(\lambda)$, $T_2^h(\lambda)$ and $T_2^v(\lambda)$. To be precise, in the first illustration removing precisely one of the cells containing a cross yields an element of~$T_1(\lambda)$. In the second illustration, either the two cells indicated by bullets or the two cells indicated by crosses can be removed to obtain an element in~$T_2^h(\lambda)$. Similarly for the third picture, removing the two cells indicated by either a cross or a bullet gives an element from~$T_2^v(\lambda)$.
 \begin{figure}[t]
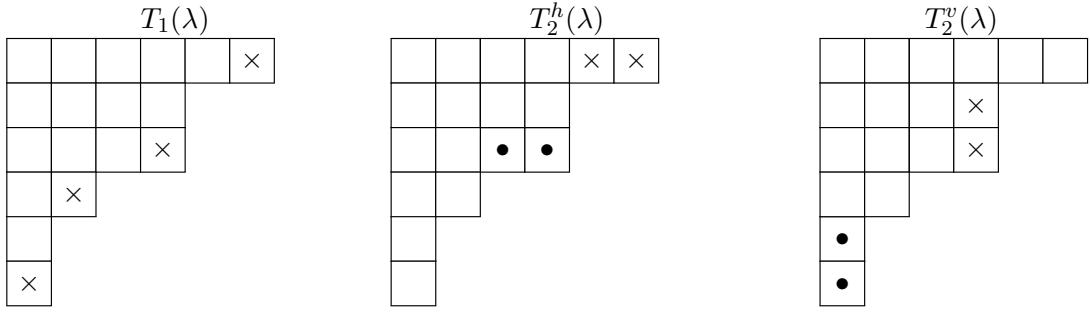

 \centering
 $T_1(\lambda)$ \hspace{4cm} $T_2^h(\lambda)$ \hspace{4cm} $T_2^v(\lambda)$ \\
 \ydiagram[*(white) \times]{5+1,4+0,3+1,1+1,1+0,0+1}*[*(white)]{5,4,3,1,1,0}
 \qquad \qquad
 \ydiagram[*(white) \times]{4+2,4+0,4+0,2+0,1+0,1+0}*[*(white)]{4,4,2,2,1,1}*[*(white) \bullet]{6+0,4+0,2+2,2+0,1+0,1+0}
 \qquad \qquad
 \ydiagram[*(white) \times]{6+0,3+1,3+1,2+0,1+0,1+0}*[*(white)]{6,3,3,2,0,0}*[*(white) \bullet] {6+0,4+0,4+0,2+0,0+1,0+1}\\
 \caption{Indication of $T_1(\lambda)$, $T_2^h(\lambda)$ and $T_2^v(\lambda)$ for the partition $\lambda=(6,4,4,2,1,1)$.}
 \label{fig:1}
 \end{figure}
 Hence the sets of partitions are given by
 \begin{gather*}
 T_1(\lambda) = \{(5,4,4,2,1,1),(6,4,3,2,1,1),(6,4,4,1,1,1),(6,4,4,2,1)\}, \\
 T_2^h(\lambda) = \{(4,4,4,2,1,1),(6,4,2,2,1,1)\}, \\
T_2^v(\lambda) = \{(6,3,3,2,1,1),(6,4,4,2)\}.
 \end{gather*}
 Note that not all elements of $T_1(\lambda)$ and $\tilde{T}_2(\lambda)$ have the same length.
\end{Example}

\section{Main results}\label{sectionmainresults}

We state our new results in Section~\ref{sec:NewResults}. We postpone the proofs until later sections. In Section~\ref{sec:Corollaries}, we give some corollaries of the generating recurrence relation. Some of these corollaries are already known in the existing literature but now have easy direct proofs using our recurrence relation.

\subsection{New results}\label{sec:NewResults}
Our main result is a generalization of the defining recurrence relation for Hermite polyno\-mials~\eqref{eq:recursionforstandardhermite} to Wronskian Hermite polynomials. The proof is given in Section~\ref{sectionproofrecursion}.

\begin{Theorem}[generating recurrence relation]\label{thm:RecurrenceRelation}
 Let $\lambda$ be a partition such that $\lvert \lambda \rvert \geq 2$. Then
 \begin{gather}\label{eq:RecurrenceRelationHe}
 F_{\lambda} \He_{\lambda}(x)  = x \sum_{\mu\in T_1(\lambda)} F_{\mu} \He_{\mu}(x) + (|\lambda|-1) \sum_{\rho\in \tilde{T}_2(\lambda)} \sgn(\rho,\lambda) F_{\rho} \He_{\rho}(x),
 \end{gather}
 where $T_1(\lambda)$, $\tilde{T}_2(\lambda)$ and $\sgn(\rho,\lambda)$ are as defined in Section~{\rm \ref{subsectionnotation}}.
\end{Theorem}

The above recurrence relation generates all Wronskian Hermite polynomials using the fact that we know that
\begin{gather*}
\He_{\varnothing}(x)=\He_0(x)=1, \qquad \He_{(1)}=\He_1(x)=x.
\end{gather*}

We also have the following recurrence relation, which is not generating but gives a lot of structure. The proof of this theorem is given in Section~\ref{sectiontopdownrecurrence} and uses the generating recurrence relation.

\begin{Theorem}[top-down recurrence relation]\label{thm:topdownrecurrence} For any partition $\lambda$ we have
 \begin{gather}\label{eq:topdownrecurrenceforHe}
 x (\lvert \lambda\rvert+1) F_\lambda \He_\lambda(x)=\sum_{\substack{\gamma \\ \lambda \in T_1(\gamma)}} F_\gamma \He_\gamma(x),
 \end{gather}
 where $T_1(\gamma)$ is as defined in Section~{\rm \ref{subsectionnotation}}.
\end{Theorem}

We give an explicit example to illustrate the results in Theorems~\ref{thm:RecurrenceRelation} and~\ref{thm:topdownrecurrence}.

\begin{Example}
 Take $\lambda=(4,2,1)$. We have
 \begin{gather*}
 35 \He_{(4,2,1)}(x) = x \left( 16 \He_{(3,2,1)}(x) + 10 \He_{(4,1,1)}(x) + 9 \He_{(4,2)}(x) \right)
 - 30 \He_{(2,2,1)}(x).
 \end{gather*}
 and
 \begin{gather*}
 280x \He_{(4,2,1)}(x)= 64 \He_{(5,2,1)}(x) + 70 \He_{(4,3,1)}(x) +56 \He_{(4,2,2)}(x) + 90 \He_{(4,2,1,1)}(x).
 \end{gather*}
\end{Example}

A remarkable consequence of the generating recurrence relation is that the average polynomial, with respect to the Plancherel measure, of all Wronskian Hermite polynomials of a fixed degree is the monomial of that degree.

\begin{Theorem}\label{thm:Average} Suppose that $n$ is a non-negative integer. Then
\begin{gather*}
\sum_{\lambda\vdash n} \frac{F_\lambda^2}{n!} \He_{\lambda}(x)= x^n.
\end{gather*}
\end{Theorem}

The observation of the statement in this theorem was in fact the starting point for the research conducted for this paper. In Section~\ref{sectionproofaverage} we prove this claim inductively by using the generating recurrence relation.

\subsection{Other corollaries of the generating recurrence relation}\label{sec:Corollaries}
The recurrence relation for Hermite polynomials \eqref{eq:recursionforstandardhermite} has remarkable consequences for the Hermite polynomials itself, namely:
\begin{itemize}\itemsep=0pt
 \item[(a)] The derivative of the $n^\textrm{th}$ Hermite polynomial is given in terms of the Hermite polynomial of a degree lower \cite[formula~(22.8.8)]{Abramowitz_Stegun}, i.e.,
 \begin{gather}\label{eq:DerHe}
 \He'_n(x)=n\He_{n-1}(x).
 \end{gather}
 \item[(b)] $\He_n$ is an even (odd) polynomial if the degree $n$ is even (odd).
 \item[(c)] $\He_n$ is a monic polynomial with integer coefficients.
\end{itemize}
Note that all of these properties can be deduced directly from using the recurrence relation \eqref{eq:recursionforstandardhermite}. Now, we generalize these properties to the setting of Wronskian Hermite polynomials. First, we deal with statement (a): the derivative.

\begin{Proposition}\label{prop:OmegaDer} For any partition $\lambda$ we have
 \begin{gather*}
 \left(F_\lambda \He_{\lambda}\right)'  = |\lambda| \sum_{\mu\in T_{1}(\lambda)} F_{\mu} \He_{\mu}.
 \end{gather*}
\end{Proposition}

Up to a combinatorial lemma, this statement follows from the generating recurrence relation by induction. The full proof is given in Section~\ref{sectionproofderivative}.

Applying the result in Proposition \ref{prop:OmegaDer} multiple times, we arrive at an expression for the higher-order derivatives. This yields a multiple sum expression, which can be interpreted in terms of skew-Young diagrams~\cite{Bona}.

Next, we consider the evenness and oddness of these polynomials as in statement (b).

\begin{Lemma}\label{lem:Omega-x}For any partition $\lambda$ we have
 \begin{gather*}
 \He_{\lambda}(-x) = (-1)^{|\lambda|} \He_{\lambda}(x).
 \end{gather*}
\end{Lemma}

In \cite[formula~(10)]{GomezUllate_Grandati_Milson-14} and \cite[Section~2]{Felder_Hemery_Veselov}, this result was already shown, but now it follows directly from the generating recurrence relation by induction. As a direct consequence of Lemma~\ref{lem:Omega-x}, we see that these polynomials are either even or odd, depending on their degree.

It can be seen directly from the generating recurrence relation that the Wronskian Hermite polynomials have rational coefficients. However, we have an even stronger conjecture which is related to statement (c) for the ordinary Hermite polynomials.

\begin{Conjecture}\label{conj:integercoefficients} For every partition $\lambda$ we have $\He_{\lambda}(x) \in \mathbb{Z}[x]$.
\end{Conjecture}

We have verified this conjecture by using the computer software Maple for every partition $\lambda$ with degree $\lvert \lambda \rvert \leq 35$.

Besides the generalization of the properties of the standard Hermite polynomials, there are also properties of the Wronskian Hermite polynomials that only make sense in this setting. For example, there is an intimate relation between the polynomial associated to a partition $\lambda$ and the one associated to its conjugate $\lambda'$.

\begin{Proposition}\label{prop:OmegaConj}
 Let $\lambda$ be a partition and denote its conjugated partition by $\lambda'$. Then
 \begin{gather*}
 \He_{\lambda}(x) = i^{|\lambda|}\He_{\lambda'}(-ix).
 \end{gather*}
\end{Proposition}

This result was essentially already proven in \cite[Theorem 1.2]{Curbera_Duran}, but is now an easy corollary of the generating recurrence relation; indeed, one proves this claim inductively and uses the relation in~\eqref{eq:verticalhorizontalconjugate}.

\section{Notes and remarks}\label{sectionnotesandremarks}
In this section we make a few notes and remarks on the results described above. First of all, we discuss our results in terms of classical Hermite polynomials. Subsequently, we provide some thoughts on the derivatives of Wronskian Hermite polynomials. Furthermore, we describe how our results should be interpreted for some special cases of Wronskians Hermite polynomials that are studied in the existing literature. We end this section by a~brief discussion of numerical implementation of the generating recurrence relation and a description of how this recurrence formula was obtained.

\subsection{The results in terms of classical Hermite polynomials}\label{classicalHermite}

In the preceding sections, we stated our main results on Wronskian Hermite polynomials. These polynomials are defined in Definition~\ref{def:WronskianHermite} using the $\He$ polynomials. In the literature, Wronskian Hermite polynomials are usually defined using the \textit{classical Hermite polynomials}, i.e., the $H$ polynomials. As the~$\He$ polynomials are just a rescaling of the classical $H$ polynomials, see \eqref{eq:HandHe} below, all our results can easily be translated in terms of $H$ polynomials.

\smallskip

{\bf Definition with $\boldsymbol{H}$ polynomials.}
The classical Hermite polynomials are defined \cite[formula~(22.7.13)]{Abramowitz_Stegun} by putting $H_0(x)=1$, $H_1(x)=2x$ and subsequently
\begin{gather}\label{eq:recursionforclassicalhermite}
H_n(x)=2xH_{n-1}(x) - \frac{n-1}{2} H_{n-2}(x), \qquad n\geq 2.
\end{gather}
If we compare this recurrence relation with the recurrence relation for $\He$ polynomials~\eqref{eq:recursionforstandardhermite}, we find the identity
\begin{gather}\label{eq:HandHe}
H_n(x)=2^{\frac{n}{2}} \He_n\big(\sqrt{2}x\big),
\end{gather}
for all $n$, see also \cite[formula~(22.5.19)]{Abramowitz_Stegun}. Similar to Definition \ref{def:WronskianHermite}, we define \textit{classical Wronskian Hermite polynomials} as
\begin{gather}\label{eq:OmegaH}
H_{\lambda}:=\frac{\Wr [H_{n_1},\dots,H_{n_r}]}{\Delta(n_\lambda)}.
\end{gather}
If we apply Lemma \ref{lem:WronskianPolynomial}, we get that $H_\lambda$ is a polynomial of degree $|\lambda|$ and has leading coefficient $2^{|\lambda|+\frac{r(r-1)}{2}}$. Indeed, as $H$ polynomials are not monic, the leading coefficient of $H_{\lambda}$ is obtained by combining all leading coefficients of the Hermite polynomials in the Wronskian of \eqref{eq:OmegaH}. Furthermore, taking the trivial partition $\lambda=(n)$, we have that $H_\lambda=H_n$. Finally, using elementary properties of Wronskians and \eqref{eq:HandHe}, we have that for any partition $\lambda$, it holds that
\begin{align}\label{eq:OmegaHandOmegaHe}
H_{\lambda}(x)
= 2^{\frac{|\lambda|+r(r-1)}{2}} \He_{\lambda}\big(\sqrt{2}x\big), \qquad x \in \mathbb{C}.
\end{align}
By this relation we see that all results concerning $\He_{\lambda}$ polynomials can be translated to $H_{\lambda}$ polynomials.

\smallskip

{\bf Results in terms of $\boldsymbol{H}$ polynomials.}
In Section~\ref{sectionmainresults}, we stated the generating recurrence relation for Wronskian Hermite polynomials, see Theorem~\ref{thm:RecurrenceRelation}. Since the prefactor in~\eqref{eq:OmegaHandOmegaHe} depends on~$r$, the length of the partition $\lambda$, we see that combining the above relation bet\-ween~$\He_\lambda$ and~$H_\lambda$ with the generating recurrence relation does yield a recurrence relation, but the recurrence coefficients depend on the length of the involved partitions. Indeed, not all partitions that appear in the generating recurrence relation have the same length. These complications are circumvented when one studies the monic version of the polynomials~$H_\lambda$, i.e.,
\begin{gather*}
\widehat{H}_{\lambda}:=\frac{H_\lambda}{2^{|\lambda|+\frac{r(r-1)}{2}}}.
\end{gather*}
These polynomials are certainly not generalizations of the classical Hermite polynomials, and certainly do not have integer coefficients, i.e., they violate the analogous statement of Conjecture~\ref{conj:integercoefficients}. However, if we rewrite~\eqref{eq:OmegaHandOmegaHe} in terms of $\widehat{H}_{\lambda}$, we get
\begin{gather*}
\widehat{H}_{\lambda}(x)
= 2^{-\frac{|\lambda|}{2}} \He_{\lambda}\big(\sqrt{2}x\big), \qquad x \in \mathbb{C},
\end{gather*}
which can be used to translate our results for the polynomials $\He_\lambda$ to the polynomials $\widehat{H}_\lambda$.
\begin{enumerate}\itemsep=0pt
 \item The generating recurrence relation in Theorem \ref{thm:RecurrenceRelation}:
 \begin{gather*}
 F_{\lambda} \widehat{H}_{\lambda}(x) = x \sum_{\mu\in T_1(\lambda)} F_{\mu} \widehat{H}_{\mu}(x)
 + \frac{|\lambda|-1}{2} \sum_{\rho\in \tilde{T}_2(\lambda)} \sgn(\rho,\lambda) F_{\rho} \widehat{H}_{\rho}(x).
 \end{gather*}
 Together with the constraints
 \begin{gather*}
 \widehat{H}_\varnothing=H_0\equiv 1, \qquad \widehat{H}_{(1)}(x)=\tfrac{1}{2} H_1(x)=x,
 \end{gather*}
 this generates all the polynomials $\widehat{H}_\lambda$.

 \item The top-down recurrence relation in Theorem \ref{thm:topdownrecurrence}:
 \begin{gather*}
 x (\lvert \lambda\rvert+1) F_\lambda \widehat{H}_\lambda(x)=\sum_{\substack{\gamma \\ \lambda \in T_1(\gamma)}} F_\gamma \widehat{H}_\gamma(x).
 \end{gather*}
 \item The average in Theorem \ref{thm:Average}:
 \begin{gather*}
 \sum_{\lambda\vdash n} \frac{F_\lambda^2}{n!} \widehat{H}_{\lambda}(x) = x^n.
 \end{gather*}
\end{enumerate}
The other corollaries of the generating recurrence relation, i.e., Proposition~\ref{prop:OmegaDer}, Lemma~\ref{lem:Omega-x} and Proposition~\ref{prop:OmegaConj}, also hold when replacing the polynomial $\He_\lambda$ by $\widehat{H}_\lambda$. Regarding Conjecture~\ref{conj:integercoefficients} in terms of $H$ polynomials, we believe that $H_\lambda\in \mathbb{Z}[x]$ for any partition $\lambda$, which we checked by the computer software Maple up to all partitions~$\lambda$ such that $|\lambda|\leq35$. It is easy to see that $\widehat{H}_\lambda \not\in \mathbb{Z}[x]$, take for example the trivial partition $\lambda=(n)$, with $n > 1$.

The authors believe that it is more natural to consider the results in this article for the monic polynomials~$\He_\lambda$. Indeed, the polynomials~$\He_\lambda$ generalize the ordinary Hermite polynomials~$\He_n$ and are expected to have integer coefficients. When working with the classical Hermite polynomials, one has to make the distinction between the non-monic polynomials~$H_\lambda$ and the monic polynomials~$\widehat{H}_\lambda$.

\subsection{Derivatives of Wronskian Hermite polynomials}\label{subsectionnotesderivative}
In Proposition \ref{prop:OmegaDer} we provided an expression for the derivative of the Wronskian Hermite polynomial associated to the partition $\lambda$ in terms of the set $T_1(\lambda)$. There are a few remarks to be made.

\smallskip

{\bf The appearance of the derivative in the generating recurrence relation.}
Combining Theorem \ref{thm:RecurrenceRelation} and Proposition \ref{prop:OmegaDer} yields that we have \begin{gather}\label{eq:RecurrenceWithDerivative}
 F_{\lambda} \He_{\lambda}(x)
 = \frac{x}{|\lambda|} F_{\lambda} \He'_{\lambda}(x)
 + (|\lambda|-1) \sum_{\rho\in \tilde{T}_2(\lambda)} \sgn(\rho,\lambda) F_{\rho} \He_{\rho}(x).
\end{gather}
This should be compared to the alternative form of the recurrence for the ordinary Hermite polynomials
\begin{gather*}
 \He_n(x)=\frac{x}{n}\He'_{n}(x)-(n-1)\He_{n-2}(x).
\end{gather*}
Identity \eqref{eq:RecurrenceWithDerivative} might be useful for tackling the Veselov conjecture. This open conjecture states that $\He_{\lambda}(x)$ only has simple zeros, except possibly at the origin \cite{Felder_Hemery_Veselov}. The (possibly zero) multiplicity of the zero at the origin is well-understood, see for example \cite{GarciaFerrero_GomezUllate}.

\smallskip

{\bf The first derivative.} The proof of Proposition \ref{prop:OmegaDer} in Section~\ref{sectionproofderivative} uses the generating recurrence relation and a combinatorial lemma. However, one can also see in a direct manner that the derivative of the Wronskian associated to the partition $\lambda$ can be described by the set $T_1(\lambda)$. Namely, observe that a determinant is multi-linear in the columns, whence the derivative of the Wronskian of a set of polynomials is a sum of Wronskians, i.e.,
\begin{gather*}
\left(\Wr[P_{n_1},P_{n_2},\dots,P_{n_r}]\right)'=\sum_{i=1}^{r} \Wr[P_{n_1},\dots,P_{n_{i-1}}, P'_{n_{i}},P_{n_{i+1}},\dots,P_{n_r}],
\end{gather*}
for polynomials $P_{n_1},P_{n_2},\dots,P_{n_r}$. Since we are working with Hermite polynomials (which satis\-fy~\eqref{eq:DerHe}), in our case each term in this sum has exactly the same degree sequence as the original degree sequence except for one coordinate. In this coordinate the degree is precisely one smaller. If one carries out the calculations, one indeed sees the coefficients as stated in Proposition~\ref{prop:OmegaDer}.

\smallskip

{\bf Uniqueness of the coefficients.} In Proposition \ref{prop:OmegaDer} we claim two things. Firstly, the derivative of the Wronskian associated to the partition $\lambda$ is in the span of the polynomials associated to the $\mu\in T_1(\lambda)$. Secondly, we provide explicit coefficients for a linear combination. However, we have not claimed that the polynomials associated to the $\mu\in T_1(\lambda)$ are linearly independent. Hence, there might be other coefficients such that the same result applies, although we do conjecture that these polynomials are indeed linearly independent for all $\lambda$. We have checked this conjecture by using the computer software Maple for all partitions of degree $\lvert \lambda \rvert \leq 35$.

\subsection{Special cases of Wronskian Hermite polynomials}\label{subsectionnotesspecialcases}

{\bf Hermite polynomials.} The ordinary Hermite polynomials correspond to the trivial partitions $\lambda=(n)$. One can check that the generating recurrence relation in Theorem~\ref{thm:RecurrenceRelation} reduces to the well-known three term recurrence relation~\eqref{eq:recursionforstandardhermite} (or~\eqref{eq:recursionforclassicalhermite}) by noting that $T_1((n))=\{(n-1)\}$, $T_2^h((n))=\{(n-2)\}$ and $T_2^v((n))=\varnothing$. Likewise, all the other results discussed in Section~\ref{sec:Corollaries} reduce to the well-known properties of Hermite polynomials, except for the top-down recurrence relation. Indeed, we have that the partition $(n)$ is also in $T_1((n,1))$ and the Wronskian Hermite polynomial associated to the partition $(n,1)$ does not have a classical counterpart.

\smallskip

{\bf Rational solutions of the fourth Painlev\'e equation.} In the literature \cite{Clarkson-3,Clarkson-1,Clarkson-2,Noumi_Yamada,Okamoto,VanAssche} Wronskian Hermite polynomials have been studied for their connection to the fourth Painlev\'e equation. Those polynomials are in fact special cases of the Wronskian Hermite polynomials defined in this paper, i.e., they correspond to the choice of specific partitions. We discuss two of such specific partitions.

The first example is the case where the Young diagram of the partition $\lambda$ has a rectangular shape, i.e., $\lambda\vdash rn$ with $\lambda=(n,n,\dots,n)$. These polynomials are discussed in \cite[Section~4]{Clarkson-2} and are often called generalized Hermite polynomials. Moreover, for these specific partitions, the Wronskian can be written as a Hankel determinant, see for example equation (6.12) in \cite{VanAssche}. We have the following result due to the generating recurrence relation.

\begin{Proposition}
 Let $\lambda=(n,n,\dots,n)$ be a partition of length $r\geq1$ with $n\geq3$. Then
 \begin{gather*}
 \He_{\lambda}(x) = x \cdot \He_{(n,\dots,n,n-1)}(x) + \frac{(r-1)(n+1)}{2} \He_{(n,\dots,n,n-1,n-1)}(x) \\
\hphantom{\He_{\lambda}(x) =}{} - \frac{(r+1)(n-1)}{2} \He_{(n,\dots,n,n-2)}(x).
 \end{gather*}
\end{Proposition}
\begin{proof}
 For a rectangular shaped Young diagram, we have
 \begin{gather*}
 T_1(\lambda)=\{(n,\dots,n,n-1)\}, \\
 \tilde{T}_2(\lambda)=\{(n,\dots,n,n-1,n-1),(n,\dots,n,n-2)\}.
 \end{gather*}
 The corresponding coefficients can be easily checked by using \eqref{eq:Flambda}.
\end{proof}

We remark that the previous recurrence formula does not describe the Wronskian Hermite polynomials associated to rectangular partitions in terms of other rectangular partitions; i.e., this class is not stable under the recurrence relation. The result for Wronskians Hermite polynomials using the $H$ polynomials, is given by
\begin{gather*}
\widehat{H}_{\lambda}(x) = x \cdot \widehat{H}_{(n,\dots,n,n-1)}(x)+ \frac{(r-1)(n+1)}{4} \widehat{H}_{(n,\dots,n,n-1,n-1)}(x)\\
\hphantom{\widehat{H}_{\lambda}(x) =}{} - \frac{(r+1)(n-1)}{4} \widehat{H}_{(n,\dots,n,n-2)}(x).
\end{gather*}

The second example of specific Wronskian Hermite polynomials that appear in rational solutions of the Painlev\'e $IV$ equation are the generalized Okamoto polynomials \cite{Clarkson-2,Noumi_Yamada,Okamoto}. These Wronskian Hermite polynomials correspond to partitions of one of the two following types; either
\begin{gather*}
\lambda =(m+2n,m+2n-2,\dots,m+2,m,m,m-1,m-1,\dots,1,1),
\end{gather*}
or
\begin{gather*}
\lambda =(m+2n-1,m+2n-3,\dots,m+1,m,m,m-1,m-1,\dots,1,1),
\end{gather*}
where $m$ is a positive integer and $n$ is a non-negative integer. In both cases, the elements in $\tilde{T}_2(\lambda)$ are easily described, but again, the obtained recurrence relation does not express generalized Okamoto polynomials in terms of generalized Okamoto polynomials of lower degree.

\smallskip

{\bf Triangular Young diagram.} The last special case that we cover here is the case of \emph{staircase} partitions; these are partitions of the form
\begin{gather*}
\lambda=(r,r-1,r-2,\dots,1).
\end{gather*}
A special feature of these partitions is that $\tilde{T}_2(\lambda)=\varnothing$. Moreover, those partitions are the only ones with this property. The set $T_1(\lambda)$ is maximal and consists of $r$ elements. As a final remark, we note that these polynomials have a straightforward explicit expression.

\begin{Lemma}\label{lem:OmegaTrap}
 For any positive integer $r$ we have
 \begin{gather*}
 \He_{(r,r-1,\dots,1)}(x)=\widehat{H}_{(r,r-1,\dots,1)}(x)= x^{\frac{r(r+1)}{2}}.
 \end{gather*}
\end{Lemma}

The proof consists of simple computations using elementary properties of determinants and is therefore omitted.

\subsection{Exceptional Hermite polynomials}
Exceptional Hermite polynomials \cite{Duran-Hermite,GomezUllate_Grandati_Milson-14,Haese-Hill_Hallnas_Veselov} appeared as a class of exceptional orthogonal polynomials. They are defined as
\begin{gather*}
H_{\lambda,n}
= \Wr[H_{n-|\lambda|+r},H_{n_1},\dots,H_{n_r}],
\qquad
n\in\mathbb{N}_{\lambda},
\end{gather*}
for any partition $\lambda$ with degree sequence $n_{\lambda}=(n_1,\dots,n_r)$ and where
\begin{gather*}
\mathbb{N}_{\lambda}
= \{ n \in \mathbb{N} \,|\, n \geq |\lambda|-r \text{ and } n-|\lambda|+r \neq n_i \text{ for all $i=1,\dots,r$} \}.
\end{gather*}
By Lemma \ref{lem:WronskianPolynomial} the polynomials $H_{\lambda,n}$ has degree $n$ for every $n\in\mathbb{N}_{\lambda}$. The set $\mathbb{N}_{\lambda}$ is cofinite, i.e., its complement within $\mathbb{N}$ is finite. The remarkable property (for \emph{even} partitions \cite{GomezUllate_Grandati_Milson-14}) is that although there is only a cofinite number of polynomials, the set $\{H_{\lambda,n}\}_{n\in \mathbb{N}_\lambda}$ is a (complete) orthogonal basis of the underlying Hilbert space. This is precisely what makes them \emph{exceptional}. An open and interesting problem within this class of polynomials is to find efficient ways to produce recurrence formulas for these polynomials. The aim of this section is to point out the implications of our main result for the theory of these polynomials. To this end, let $\widehat{H}_{\lambda,n}$ denote the monic exceptional Hermite polynomial. We can apply our generating recurrence relation to the class of exceptional Hermite polynomials. Take $n\in\mathbb{N}_{\lambda}$ such that $n-|\lambda|\geq \lambda_1+2$. Then $n-1,n-2\in\mathbb{N}_{\lambda}$. Define $(\lambda,n)$ as the partition of the elements of $\lambda$ and $n-|\lambda|$, i.e., $(\lambda,n)=(n-|\lambda|,\lambda_1,\lambda_2,\dots,\lambda_r)$. Furthermore, note that for any $\mu\in T_1(\lambda)$ and any $\rho \in \tilde{T}_2(\lambda)$, we have $n-1\in \mathbb{N}_\mu$ and $n-2\in \mathbb{N}_\rho$. Applying the generating recurrence relation gives
\begin{gather}
F_{(\lambda,n)} \widehat{H}_{\lambda,n}(x) = x F_{(\lambda,n-1)} \widehat{H}_{\lambda,n-1}(x)- \frac{n-1}{2} F_{(\lambda,n-2)} \widehat{H}_{\lambda,n-2}(x)\nonumber\\
\hphantom{F_{(\lambda,n)} \widehat{H}_{\lambda,n}(x) =}{}+ x \sum_{\mu\in T_1(\lambda)} F_{(\mu,n-1)} \widehat{H}_{\mu,n-1}(x) \nonumber\\
\hphantom{F_{(\lambda,n)} \widehat{H}_{\lambda,n}(x) =}{}+ \frac{n-1}{2} \sum_{\rho\in \tilde{T}_2(\lambda)} \sgn(\rho,\lambda) F_{(\rho,n-2)} \widehat{H}_{\rho,n-2}(x).\label{eq:RecurrenceExceptionalH}
\end{gather}
So, we see that the exceptional Hermite polynomial associated to the partition $\lambda$ of degree $n$ can be expressed in terms of a set of exceptional Hermite polynomials of degrees $n-1$ and $n-2$, but not only those in the family associated to the partition $\lambda$, but also those in the family associated to the partitions in $T_1(\lambda)$ and $\tilde{T}_2(\lambda)$.

Finally, if $n-|\lambda| < \lambda_1+2$, similar expressions as \eqref{eq:RecurrenceExceptionalH} can be derived. As they are more complicated to write down, we do not give more details here.

Our recurrence relation in terms of exceptional Hermite polynomials differs from the already known recurrence relations \cite{Duran-Recurrence,GomezUllate_Grandati_Milson-14,Gomez-Ullate_Kasman_Kuijlaars_Milson,Miki_Tsujimoto} as we approach in the general setting of Wronskian Hermite polynomials and then specify to exceptional Hermite polynomials. Therefore, we can approach purely combinatorical and we do not anything of the theory of exceptional orthogonal polynomials. The advantage of the recurrence relation presented in this article is that we provide a closed form for the recurrence coefficients.

\subsection{Implementation of the generating recurrence relation}\label{sec:Implementation}
We like to point out an upper bound for the number of elements in the generating recurrence relation, see Theorem \ref{thm:RecurrenceRelation}. Suppose that $\lambda$ is a partition of length $r$. Then, clearly, the number of elements in $T_1(\lambda)$ is bounded by $r$. Moreover, one straightforwardly sees that $|T_1(\lambda)|=|T_1(\lambda')|$. Hence $T_1(\lambda)$ is bounded by $\min\{r,\lambda_1\}$. Likewise, the number of elements in $T_2^h(\lambda)$ is also bounded by $r$. A slightly deeper analysis shows that whenever $\mu\in T_2^v(\lambda)$ is obtained by removing a cell in the Young diagram of $\lambda$ in both row $i$ and $i+1$, we have that $\lambda_i=\lambda_{i+1}$, so that we cannot remove two blocks from $\lambda$ in row $i$ to obtain an element in $T_2^h(\lambda)$. Therefore, we have that the number of elements of $\tilde{T}_2(\lambda)$ is in fact bounded by $r$. Again, then also $\lvert \tilde{T}_2(\lambda)\rvert = \lvert \tilde{T}_2(\lambda')\rvert \leq \lambda_1$, so indeed $\lvert \tilde{T}_2(\lambda)\rvert \leq \min\{r,\lambda_1\}$ Therefore, there are at most $2\min\{r,\lambda_1\}$ elements in the generating recurrence relation. For many partitions, this number is in fact still a big overestimation, but we omit further analysis.

\subsection{Obtaining the generating recurrence relation}\label{sec:Obtaining}
As mentioned in Section~\ref{sectionmainresults}, the starting point of this research was the observation that the average of all monic Wronskians of a fixed degree (with respect to the Plancherel measure) is simply the monomial of that degree. These observations were based on explicit computations of the examples for low degrees. At some point during the search for a proof of this fact, which was verified to a high degree by using the computer software Maple, the authors wondered which of the properties of standard Hermite polynomials would still hold in a generalized version.

For a generalization of the three term recurrence relation~\eqref{eq:recursionforstandardhermite}, one first has to determine the natural generalization of the concept of 'one degree smaller' and 'two degrees smaller' as in~\eqref{eq:recursionforstandardhermite} or~\eqref{eq:recursionforclassicalhermite}. A~natural thing to do is to consider all partitions which yield a polynomial that is one or two degrees smaller. A first check using the aforementioned computer software Maple, showed that using the sets $T_1(\lambda)$ and~$T_2(\lambda)$ one can indeed find a recurrence relation. In fact, there are some degrees of freedom left when one works with the bigger set~$T_2(\lambda)$ instead of~$\tilde{T}_2(\lambda)$. A~natural choice for the coefficients for~$T_1(\lambda)$, namely the ones according to the derivative as indicated in~\eqref{eq:RecurrenceWithDerivative}, then showed that one can restrict the generating recurrence relation to using~$\tilde{T}_2(\lambda)$ instead of the whole set~$T_2(\lambda)$.

\section{Proof of Theorem \ref{thm:RecurrenceRelation}: the generating recurrence relation}\label{sectionproofrecursion}

In this section we prove our first main result: the generating recurrence relation presented in Theorem~\ref{thm:RecurrenceRelation}. A close inspection of~\eqref{eq:RecurrenceRelationHe} shows that each polynomial $\He_\lambda$ appears with the prefactor $F_\lambda$. Moreover, in the definition of the Wronskian Hermite polynomial associated to the partition $\lambda$, see~\eqref{eq:OmegaHe}, we use the degree sequence $n_\lambda$ of the partition as in~\eqref{eq:DegreeSequence}. Now we define a polynomial $\Phi[n_1,\dots,n_r]$ for a broader class of finite sequence $(n_1,\dots,n_r)\in \mathbb{Z}^r$, such that in the special case that $(n_1,\dots,n_r)$ is in fact the degree sequence of some partition $\lambda$, we have
\begin{gather}\label{eq:fromHetoPhi}
F_\lambda \He_\lambda=\Phi[n_1,\dots,n_r].
\end{gather}
Before we turn to the actual proof of the generating recurrence relation, we therefore first define the polynomials $\Phi$ such that the above is satisfied, and subsequently derive some essential basic of these polynomials. These two things are both technical, but essential for the proof of the recurrence relation.

For the definition of the $\Phi$ polynomials, we first note that it is an easy undertaking to rewrite~\eqref{eq:Flambda} for $F_\lambda$ in terms of the degree sequence $n_\lambda$ of $\lambda$; indeed, we have that
\begin{gather*}
\lvert \lambda \rvert= \sum\limits_{i=1}^r (n_\lambda)_i - \frac{r(r-1)}{2}.
\end{gather*}
 Using this, we have the rather cumbersome expression
\begin{gather*}
F_\lambda=(-1)^{\frac{r(r-1)}{2}} \cdot \frac{\left(\sum\limits_{i=1}^r (n_\lambda)_i - \frac{r(r-1)}{2}\right)! \cdot \Delta(n_{\lambda})}{\prod\limits_{i=1}^{r}(n_\lambda)_i!}.
\end{gather*}
We can now generalize this to finite sequences $(n_1,\dots,n_r)$ that are not necessarily associated to partitions.

\begin{Definition} For every $r\in \mathbb{N}$ and any finite sequence $(n_1,\dots,n_r)\in \mathbb{Z}^r$, we define the number
 \begin{gather}\label{eq:defa}
 a[n_1,\dots,n_r]:=\sum_{i=1}^r n_i - \frac{r(r-1)}{2} = \sum_{i=1}^r (n_i-r+i),
 \end{gather}
 i.e., if $(n_1,\dots,n_r)$ is the degree sequence of a partition $\lambda$, then $a[n_1,\dots,n_r]=\lvert \lambda\rvert$. Next, for any finite sequence $(n_1,\dots,n_r)$ such that $a[n_1,\dots,n_r]\geq 0$, we define the polynomial
 \begin{gather}\label{eq:defPhi}
 \Phi[n_1,\dots,n_r]=(-1)^{\frac{r(r-1)}{2}} \cdot \frac{a[n_1,\dots,n_r]!}{\prod\limits_{i=1}^{r} n_i!} \Wr[\He_{n_1},\dots,\He_{n_r}].
 \end{gather}
\end{Definition}

In the above definition, we use the convention that $\He_n\equiv 0$ whenever $n<0$, and $1/k!=0$ whenever $k<0$. Comparing the definitions~\eqref{eq:OmegaHe} and~\eqref{eq:defPhi}, we now have defined the polyno\-mials~$\Phi$ in such a way that \eqref{eq:fromHetoPhi} holds if $(n_1,\dots,n_r)$ is the degree sequence of $\lambda$, i.e., we have extended the class of polynomials that we set out to study in the generating recurrence relation.

By the definition this new class of polynomials, the following properties are trivial and therefore we omit the proof.

\begin{Proposition}\label{prop:fundamentalsPhi}
 Suppose that $r\in \mathbb{N}$ and let $(n_1,\dots,n_r)\in \mathbb{Z}^r$ such that $a[n_1,\dots,n_r]\geq 0$. Then we have the following results:
 \begin{enumerate}\itemsep=0pt
 \item[$1.$] If there is an $i\in \{1,\dots,r\}$ such that $n_i<0$, then $\Phi[n_1,\dots,n_r]\equiv 0$.
 \item[$2.$] The function $\Phi$ is anti-symmetric, i.e., if $\sigma$ is a permutation of $\{1,\dots,r\}$, then
 \begin{gather*}\Phi[n_1,\dots,n_r]=\sgn(\sigma) \Phi[n_{\sigma(1)},n_{\sigma(2)},\dots,n_{\sigma(r)}].\end{gather*}
 \item[$3.$] If $n_i=n_j$ for $i\neq j$, then $\Phi[n_1,\dots,n_r]\equiv 0$.
 \end{enumerate}
\end{Proposition}

In the above proposition, all the involved degree sequences $(n_1,\dots,n_r)$ have the same length. It is also possible to reduce the length, as in the following statement.

\begin{Proposition}\label{prop:zeroinPhi}
 Suppose that $r\in \mathbb{N}$ and let $(n_1,\dots,n_r)\in \mathbb{Z}^r$ such that $a[n_1,\dots,n_r]\geq 0$. Then we have that
 \begin{gather*}\Phi[n_1,\dots,n_r,0]=\Phi[n_1-1,\dots,n_r-1].\end{gather*}
\end{Proposition}
\begin{proof}
 First note that by \eqref{eq:defa} we obtain
 \begin{gather}\label{prop:zeroinPhiproof1}
 a[n_1,\dots,n_r,0]
 = a[n_1-1,\dots,n_r-1].
 \end{gather}
 Next, we have the general Wronskian property
 \begin{gather*}
 \Wr[f_1,\dots,f_r,1]
 = (-1)^{r} \Wr[f'_1,\dots,f'_r],
 \end{gather*}
 for sufficiently differentiable functions $f_1,\dots,f_r$. Since $\He_0\equiv 1$, we can apply this property and if we invoke \eqref{eq:DerHe}, we get that
 \begin{gather}\label{prop:zeroinPhiproof2}
 \Wr[\He_{n_1},\dots,\He_{n_r},\He_0] =(-1)^{r} \left(\prod_{i=1}^r n_i \right) \cdot \Wr[\He_{n_1-1},\dots,\He_{n_r-1}].
 \end{gather}
 Combining \eqref{prop:zeroinPhiproof1} and \eqref{prop:zeroinPhiproof2}, we obtain
 \begin{align*}
 \Phi[n_1,\dots,n_r,0] &= (-1)^{\frac{(r+1)r}{2}} \cdot \frac{a[n_1,\dots,n_r,0]!}{\prod\limits_{i=1}^{r} n_i!} \Wr[\He_{n_1},\dots,\He_{n_r},\He_0]\\
 &=(-1)^{\frac{r(r-1)}{2}} \cdot \frac{a[n_1-1,\dots,n_r-1]!}{\prod\limits_{i=1}^{r} (n_i-1)!} \Wr[\He_{n_1-1},\dots,\He_{n_r-1}] \\
 &=\Phi[n_1-1,\dots,n_r-1],
 \end{align*}
 which is indeed what we intended to prove.
\end{proof}

To prove the generating recurrence relation, we expand Wronskians as sums over permutations, as can always be done with determinants. For this, denote $S_r$ as the set of all permutations of $r$ elements. We define, for any $n=(n_1,\dots,n_r)\in \mathbb{Z}^r$ such that $a[n]\geq 0$ and for any $\sigma\in S_r$, the shorthand notation
\begin{gather}\label{eq:defsigman}
\sigma(n)_i=n_i-\sigma(i)+1, \qquad i\in \{1,\dots,r\}.
\end{gather}
Since we have \eqref{eq:DerHe}, we know that calculating the entries of the determinant gives
\begin{gather*}
\Wr[\He_{n_1},\dots,\He_{n_r}]
= \begin{vmatrix}
c_{1,1} \He_{n_1} & c_{1,2} \He_{n_2} & \dots & c_{1,r} \He_{n_r} \\
c_{2,1} \He_{n_1-1} & c_{2,2} \He_{n_2-1} & \dots & c_{2,r}\He_{n_r-1} \\
\vdots & \vdots & \ddots & \vdots \\
c_{r,1} \He_{n_1-r+1} & c_{r,2} \He_{n_2-r+1} & \dots & c_{r,r}\He_{n_r-r+1} \\
\end{vmatrix},
\end{gather*}
where $c_{i,j} = \frac{n_j!}{(n_j-i+1)!}$ for $i,j=1,\dots,r$. Therefore, we get that
\begin{gather*}
\Wr[\He_{n_1},\dots,\He_{n_r}]=\sum_{\sigma\in S_r} \sgn(\sigma) \prod_{i=1}^r \frac{n_i!}{\sigma(n)_i!} \He_{\sigma(n)_i},
\end{gather*}
and hence, combining this with the defining equation \eqref{eq:defPhi}, we see that
\begin{gather}\label{eq:Phioverpermutations}
\Phi[n_1,\dots,n_r]=(-1)^{\frac{r(r-1)}{2}} a[n_1,\dots,n_r]! \sum_{\sigma\in S_r} \sgn(\sigma) \prod_{i=1}^r \frac{\He_{\sigma(n)_i}}{\sigma(n)_i!}.
\end{gather}
In the following, we work with the polynomials $\Phi[n_1,\dots,n_r]$ in this form. Now we have the set-up for our proof, we turn to the first step: for any $n=(n_1,\dots,n_r)\in \mathbb{Z}^r$ and any $j\in \{1,\dots,r\}$, we define
\begin{gather}\label{eq:nj}
n[j]=(n_1,\dots,n_{j-1},n_j-1,n_{j+1},\dots,n_r).
\end{gather}
We then have the following lemma.
\begin{Lemma}\label{lem:reducingn}
 Suppose that $n=(n_1,\dots,n_r)\in \mathbb{Z}^r$ such that $a[n]\geq 1$ and let $\sigma \in S_r$. Then
 \begin{enumerate}\itemsep=0pt
 \item[$1.$] $a[n[j]]=a[n]-1$ for all $j\in \{1,\dots,r\}$,
 \item[$2.$] $\sigma(n[j])_i=\sigma(n)_i$ when $i\neq j$ and $\sigma(n[j])_j=\sigma(n)_j-1$ for all $j \in \{1,\dots,r\}$,
 \item[$3.$] $\sum\limits_{j=1}^r \sigma(n)_j=a[n]$,
 \item[$4.$] $\sum\limits_{j=1}^r \frac{a[n[j]]!}{\prod\limits_{i=1}^{r} \sigma(n[j])_i!}=\frac{a[n]!}{\prod\limits_{i=1}^r \sigma(n)_i!}$.
 \end{enumerate}
\end{Lemma}
\begin{proof} The first three items are trivial computations using the defining equations~\eqref{eq:defa} \linebreak and~\eqref{eq:defsigman}. The last item follows from the first three; indeed, we find that
 \begin{gather*}
 \sum\limits_{j=1}^r \frac{a[n[j]]!}{\prod\limits_{i=1}^{r} \sigma(n[j])_i!} = \sum\limits_{j=1}^r \frac{(a[n]-1)!}{(\sigma(n)_j-1)!\prod\limits_{i\neq j} \sigma(n)_i!}
 =\frac{(a[n]-1)!}{\prod\limits_{i=1}^r \sigma(n)_i!} \sum_{j=1}^r \sigma(n)_j
=\frac{a[n]!}{\prod\limits_{i=1}^r \sigma(n)_i!}.
 \end{gather*}
 Here, we used that $\frac{1}{(m-1)!}=\frac{m}{m!}$, which holds for any $m\in \mathbb{Z}$, since we have the convention that $\frac{1}{m!}=0$ for any $m<0$.
\end{proof}

Next, note that from \eqref{eq:nj} we get
\begin{gather*}n[j][j]=(n_1,\dots,n_{j-1},n_j-2,n_{j+1},\dots,n_r),\end{gather*}
and that we therefore have the following corollary of Lemma \ref{lem:reducingn}.

\begin{Corollary}\label{cor:reducingnby2}\samepage
 Suppose $n=(n_1,\dots,n_r)\in \mathbb{Z}^r$ such that $a[n]\geq 2$. Then
 \begin{enumerate}\itemsep=0pt
 \item[$1.$] $a[n[j][j]]=a[n]-2$ for all $j\in \{1,\dots,r\}$,
 \item[$2.$] $\sigma(n[j][j])_i=\sigma(n)_i$ when $i\neq j$ and $\sigma(n[j][j])_j=\sigma(n)_j-2$ for all $j \in \{1,\dots,r\}$.
 \end{enumerate}
\end{Corollary}

We can now prove the key ingredient for proving the generating recurrence relation.

\begin{Proposition}\label{prop:recurrenceforPhi}
 Suppose that $n=(n_1,\dots,n_r)\in \mathbb{Z}^r$ such that $a[n]\geq 2$. Then we have
 \begin{gather*}
 \Phi[n](x)=x\sum_{j=1}^r \Phi[n[j]](x)-(a[n]-1)\sum_{j=1}^r \Phi[n[j][j]](x).
 \end{gather*}
\end{Proposition}
\begin{proof} Starting from \eqref{eq:Phioverpermutations} and applying item 4 from Lemma \ref{lem:reducingn}, we obtain
 \begin{align*}
 \Phi[n](x) &=(-1)^{\frac{r(r-1)}{2}} a[n]! \sum_{\sigma\in S_r} \sgn(\sigma) \prod_{i=1}^r \frac{\He_{\sigma(n)_i}}{\sigma(n)_i!}\\
 &=\sum_{j=1}^r (-1)^{\frac{r(r-1)}{2}} a[n[j]]! \sum_{\sigma\in S_r} \sgn(\sigma) \prod_{i=1}^r \frac{\He_{\sigma(n)_i}}{\sigma(n[j])_i!}.
 \end{align*}
 Now, in the $j^\textrm{th}$ term in this sum, we apply the recurrence relation \eqref{eq:recursionforstandardhermite} to $\He_{\sigma(n)_j}$ which is
 \begin{gather*}
 \He_{\sigma(n)_j} = x \He_{\sigma(n)_j-1}(x) - (\sigma(n)_j-1)\He_{\sigma(n)_j-2}(x).
 \end{gather*}
 We obtain
 \begin{gather*}
 \Phi[n](x)= x\sum_{j=1}^r (-1)^{\frac{r(r-1)}{2}} a[n[j]]! \sum_{\sigma\in S_r} \sgn(\sigma) \He_{\sigma(n)_j-1}(x) \frac{\prod\limits_{i\neq j} \He_{\sigma(n)_i}(x)}{\prod\limits_{i=1}^r \sigma(n[j])_i!} \\
\hphantom{\Phi[n](x)=}{} -\sum_{j=1}^r (-1)^{\frac{r(r-1)}{2}} a[n[j]]! \sum_{\sigma\in S_r} \sgn(\sigma)(\sigma(n)_j-1)\He_{\sigma(n)_j-2}(x) \frac{\prod\limits_{i\neq j} \He_{\sigma(n)_i}(x)}{\prod\limits_{i=1}^r \sigma(n[j])_i!}.
 \end{gather*}
 Rewriting this equality by using the second item of Lemma~\ref{lem:reducingn} and Corollary~\ref{cor:reducingnby2} gives
 \begin{gather*}
 \Phi[n](x)= x\sum_{j=1}^r (-1)^{\frac{r(r-1)}{2}} a[n[j]]! \sum_{\sigma\in S_r} \sgn(\sigma) \prod_{i=1}^r\frac{\He_{\sigma(n[j])_i}(x)}{\sigma(n[j])_i!} \\
\hphantom{\Phi[n](x)=}{} -(a[n]-1)\sum_{j=1}^r (-1)^{\frac{r(r-1)}{2}} a[n[j][j]]! \sum_{\sigma\in S_r} \sgn(\sigma)\prod_{i=1}^r \frac{\He_{\sigma(n[j][j])_i}(x)}{\sigma(n[j][j])_i!} \\
\hphantom{\Phi[n](x)}{} = x\sum_{j=1}^r \Phi[n[j]](x)-(a[n]-1)\sum_{j=1}^r \Phi[n[j][j]](x),
 \end{gather*}
 which establishes the identity we wanted to prove.
\end{proof}

We are now in the position to prove the generating recurrence relation \eqref{eq:RecurrenceRelationHe}.

\begin{proof}[Proof of Theorem \ref{thm:RecurrenceRelation}]
 Write $\lambda=(\lambda_1,\dots,\lambda_r)$ and let $n=(n_1,\dots,n_r)$ be the degree sequence of the partition $\lambda$ as in \eqref{eq:DegreeSequence}. Then, combining \eqref{eq:fromHetoPhi} and Proposition \ref{prop:recurrenceforPhi}, we have
 \begin{gather}\label{eq:ProofRecurrence1}
 F_\lambda \He_\lambda(x) = x\sum_{j=1}^r \Phi[n[j]](x)-(\lvert \lambda \rvert-1)\sum_{j=1}^r \Phi[n[j][j]](x),
 \end{gather}
 since $a[n]=\lvert \lambda \rvert$. All that is left to prove is that the first sum is precisely the first sum in \eqref{eq:RecurrenceRelationHe} and the second sum is (up to a minus sign) the second sum in \eqref{eq:RecurrenceRelationHe}.

 We start with the first sum. We show that for every $j\in \{1,\dots,r\}$, we either have that $\Phi[n[j]]=F_\mu \He_\mu$ for some $\mu\in T_1(\lambda)$ by \eqref{eq:fromHetoPhi}, or $\Phi[n[j]]=0$. Note that there are four possibilities for $j\in \{1,\dots,r\}$:
 \begin{enumerate}\itemsep=0pt
 \item[1.] $j<r$ and $\lambda_j\neq \lambda_{j+1}$,
 \item[2.] $j=r$ and $\lambda_r>1$,
 \item[3.] $j=r$ and $\lambda_r=1$,
 \item[4.] $j<r$ and $\lambda_j=\lambda_{j+1}$.
 \end{enumerate}
 In the first and second case, $n[j]$ is the degree sequence of an element $\mu\in T_1(\lambda)$, namely
 \begin{gather*}\mu=(\lambda_1,\dots,\lambda_{j-1},\lambda_j-1,\lambda_{j+1},\dots,\lambda_r),\end{gather*}
 and hence $\Phi[n[j]]=F_\mu \He_\mu$ by \eqref{eq:fromHetoPhi}. In the third case, we see that
 \begin{gather*}n[r]=(n_1,\dots,n_{r-1},0),\end{gather*}
 and hence that $\Phi[n[r]]=\Phi[n_1-1,\dots,n_{r-1}-1]$ by Proposition \ref{prop:zeroinPhi}. However, the finite sequence $(n_1-1,\dots,n_{r-1}-1)$ is precisely the degree sequence of
 \begin{gather*}
 \mu=(\lambda_1,\dots,\lambda_{r-1})\in T_1(\lambda).
 \end{gather*}
 Therefore, by \eqref{eq:fromHetoPhi} we have $\Phi[n[r]]=F_{\mu} \He_{\mu}$. In the last case, we have $n_j=n_{j+1}+1$ such that $n[j]_j=n[j]_{j+1}$, and hence $\Phi[n[j]]=0$, by Proposition \ref{prop:fundamentalsPhi}. Note that the first, second and third case indeed cover all the possible elements of $T_1(\lambda)$, hence
 \begin{gather}\label{eq:ProofRecurrence2}
 \sum_{j=1}^r \Phi[n[j]]=\sum_{\mu\in T_1(\lambda)} F_{\mu} \He_{\mu}.
 \end{gather}

 For the second sum of \eqref{eq:ProofRecurrence1}, we distinguish nine possibilities for $j\in \{1,\dots,r\}$:
 \begin{enumerate}\itemsep=0pt
 \item[1.] $j<r$ and $\lambda_j>\lambda_{j+1}+1$,
 \item[2.] $j=r$ and $\lambda_r>2$,
 \item[3.] $j=r$ and $\lambda_r=2$,
 \item[4.] $j<r-1$ and $\lambda_j=\lambda_{j+1}>\lambda_{j+2}$,
 \item[5.] $j=r-1$ and $\lambda_{r-1}=\lambda_r>1$,
 \item[6.] $j=r-1$ and $\lambda_{r-1}=\lambda_r=1$,
 \item[7.] $j<r-1$ and $\lambda_j=\lambda_{j+1}=\lambda_{j+2}$,
 \item[8.] $j<r$ and $\lambda_j=\lambda_{j+1}+1$,
 \item[9.] $j=r$ and $\lambda_r=1$.
 \end{enumerate}
 These nine cases can be divided into three groups.

\smallskip

{\bf The horizontal cases.} The first 3 cases precisely yield the elements of $T_2^h(\lambda)$. Namely, in the first two cases, $n[j][j]$ is indeed the degree sequence of the partition
 \begin{gather*}\rho=(\lambda_1,\dots,\lambda_{j-1},\lambda_j-2,\lambda_{j+1},\dots,\lambda_r),\end{gather*}
 and $\rho \in T_2^h(\lambda)$, so $\Phi[n[j][j]]=F_{\rho} \He_{\rho}$ by \eqref{eq:fromHetoPhi}. The third case is again a little more subtle, since
 \begin{gather*}n[r][r]=(n_1,\dots,n_{r-1},0),\end{gather*}
 it follows that $\Phi[n[r][r]]=\Phi[n_1-1,\dots,n_{r-1}-1]$, by Proposition \ref{prop:zeroinPhi}. However, we note that $(n_1-1,\dots,n_{r-1}-1)$ is the degree sequence of the partition
 \begin{gather*}\rho=(\lambda_1,\dots,\lambda_{r-1}),\end{gather*}
 and $\rho\in T_2^h(\lambda)$, so $\Phi[n[r][r]]=F_{\rho} \He_{\rho}$ by \eqref{eq:fromHetoPhi}.

\smallskip

{\bf The vertical cases.} The fourth, fifth and sixth case together yield all the elements of $T_2^v(\lambda)$. In cases~4 and~5, we see that $n_j=n_{j+1}+1$. Using Proposition~\ref{prop:fundamentalsPhi}, we get
 \begin{align*}
 \Phi[n[j][j]] &=\Phi[n_1,\dots,n_{j-1},n_j-2,n_j-1,n_{j+2},\dots,n_r]\\
 &=-\Phi[n_1,\dots,n_{j-1},n_j-1,n_j-2,n_{j+2},\dots,n_r] \\
 &=-\Phi[n_1,\dots,n_{j-1},n_j-1,n_{j+1}-1,n_{j+2},\dots,n_r].
 \end{align*}
 Now, $(n_1,\dots,n_{j-1},n_j-1,n_{j+1}-1,n_{j+2},\dots,n_r)$ is precisely the degree sequence of
 \begin{gather*}\rho=(\lambda_1,\dots,\lambda_{j-1},\lambda_j-1,\lambda_{j+1}-1,\lambda_{j+2},\dots,\lambda_r),\end{gather*}
 and in fact $\rho\in T_2^v(\lambda)$, so we have $\Phi[n[j][j]]=-F_\rho \He_\rho$ by \eqref{eq:fromHetoPhi}. For the sixth case, we have $n_{r-1}=2$ and $n_r=1$, such that
 \begin{gather*}n[j][j]=(n_1,\dots,n_{r-2},0,1).\end{gather*}
 Again, applying Proposition \ref{prop:fundamentalsPhi} and Proposition \ref{prop:zeroinPhi} twice, we obtain
 \begin{align*}
 \Phi[n[j][j]]& =-\Phi[n_1,\dots,n_{r-2},1,0)]
 =-\Phi[n_1-1,\dots,n_{r-2}-1,0] \\
 & =-\Phi[n_1-2,\dots,n_{r-2}-2].
 \end{align*}
 This sequence $(n_1-2,\dots,n_{r-2}-2)$ is the degree sequence of the partition
 \begin{gather*}\rho=(\lambda_1,\dots,\lambda_{r-2}),\end{gather*}
 and indeed, $\rho\in T_2^v(\lambda)$, such that $\Phi[n[j][j]]=-F_\rho \He_\rho$ by \eqref{eq:fromHetoPhi}.

\smallskip

 {\bf The vanishing cases.} The last three cases from the above list of nine are the cases that do not have a contribution to the sum $\sum\limits_{j=1}^r \Phi[n[j][j]]$; we treat them one by one.

 In case 7, we have $j<r-1$ and $\lambda_{j}=\lambda_{j+1}=\lambda_{j+2}$, so $n_{j+1}=n_j-1$ and $n_{j+2}=n_{j}-2$, whence by Proposition \ref{prop:fundamentalsPhi}, we have
 \begin{gather*}\Phi[n[j][j]] = \Phi[n_1,\dots,n_{j-1},n_j-2,n_j-1,n_j-2,n_{j+3},\dots,n_r]=0.\end{gather*}

 In case 8, we have $j<r$ such that $\lambda_j=\lambda_{j+1}+1$. Then we have $n_{j+1}=n_j-2$, such that $n[j][j]_{j+1}=n[j][j]_j$ and hence $\Phi[n[j][j]]=0$, again by Proposition \ref{prop:fundamentalsPhi}.

 In the last case, we have $j=r$ and $\lambda_r=1$. Then $n_r=1$, so $n[r][r]=-1$. Then, by the first item of Proposition \ref{prop:fundamentalsPhi}, we have that $\Phi[n[r][r]]=0$.

 Now, combining the nine cases, we see that
 \begin{gather}
 \sum_{j=1}^r \Phi[n[j][j]] = \sum_{\rho \in T_2^h(\lambda)} F_\rho \He_\rho - \sum_{\rho \in T_2^v(\lambda)} F_\rho \He_\rho
 =-\sum_{\rho\in \tilde{T_2}(\lambda)} \sgn(\rho,\lambda) F_\rho \He_\rho. \label{eq:ProofRecurrence3}
 \end{gather}

 Finally, if we plug in \eqref{eq:ProofRecurrence2} and \eqref{eq:ProofRecurrence3} into \eqref{eq:ProofRecurrence1}, we obtain identity \eqref{eq:RecurrenceRelationHe}.
\end{proof}

\section{Proof of Theorem \ref{thm:topdownrecurrence}: the top-down recurrence relation}\label{sectiontopdownrecurrence}

To prove the top-down recurrence relation in Theorem~\ref{thm:topdownrecurrence}, we first state a combinatorial lemma and some known results. The proof of the lemma is postponed to Appendix~\ref{appendixcombinatoriallemmas}. By convention, we set an empty sum equal to zero.

\begin{Lemma}\label{lem:lemmafortopdown}
 Suppose that $\lambda$ and $\mu$ are partitions such that $\lvert \mu \rvert=\lvert \lambda \rvert -1$. Then we have that
 \begin{gather}\label{eq:lemmafortopdown}
 \sum_{\substack{\rho \\ \lambda \in T_1(\rho) \\ \mu \in \tilde{T}_2(\rho)}} \sgn(\mu,\rho) =\sum_{\gamma \in \tilde{T}_2(\lambda) \cap T_1(\mu)} \sgn(\gamma,\lambda),
 \end{gather}
 where the $\sgn$ function, $T_1(\rho)$ and $\tilde{T}_2(\rho)$ are as defined in Section~{\rm \ref{subsectionnotation}}.
\end{Lemma}

It was proven in \cite{Stanley} that the partially ordered set of partitions (as introduced in the preliminaries) forms a so-called \emph{$1$-differential poset}. This means that for any partition $\lambda$ such that $\lvert \lambda \rvert\geq 1$, we have that
\begin{gather}\label{eq:parentsandchilderen}
\# \{\rho \,|\, \lambda\in T_1(\rho)\} = 1+\# T_1(\lambda).
\end{gather}
Furthermore, for two partitions $\lambda \neq \mu$ such that $\lvert\lambda\rvert=\lvert \mu\rvert\geq 1$, we have that
\begin{gather}\label{eq:siblings}
\# \{\rho \vdash \lvert \lambda \rvert -1 \,|\, \rho \in T_1(\lambda) \cap T_1(\mu)\} =\# \{\rho \vdash \lvert \lambda\rvert+1 \,|\, \lambda,\mu\in T_1(\rho)\} \in \{0,1\}.
\end{gather}
For a fixed partition $\lambda$, we denote the set of all partitions $\mu \vdash \lvert \lambda \rvert$ such that $\mu \neq \lambda$ and the number in \eqref{eq:siblings} is 1 by $S(\lambda)$, i.e.,
\begin{gather}\label{def:Slambda}
S(\lambda)=\{\mu \neq \lambda \,|\, T_1(\lambda) \cap T_1(\mu) \neq \varnothing\}.
\end{gather}
We use these facts together with Lemma \ref{lem:lemmafortopdown} to prove the top-down recurrence relation.

\begin{proof}[Proof of Theorem \ref{thm:topdownrecurrence}]
 We approach by induction on $n:=\lvert \lambda \rvert$.

 For $n=0$, the claim is trivial, since $\He_\varnothing\equiv 1$ and $\He_{(1)}(x)=x$. For $n=1$, note that $\He_{(2)}(x)=x^2-1$ and $\He_{(1,1)}(x)=x^2+1$, and hence
 \begin{gather*}\He_{(1,1)}(x)+\He_{(2)}(x)=2x^2=2x\He_{(1)},\end{gather*}
 which establishes the claim for $n=1$.

 For the induction step, we take $n\geq 2$ and suppose that we have proven the claim for all partitions $\mu$ such that $\lvert \mu \rvert<n$. For every partition $\gamma$, we write $\psi_\gamma:=F_\gamma \He_\gamma$ to shorten the notation. Then, by applying the generating recurrence relation \eqref{eq:RecurrenceRelationHe} to all terms in the following sum, we obtain
 \begin{gather}\label{eq:expansiontopdown}
 \sum_{\substack{\rho \\ \lambda\in T_1(\rho)}} \psi_\rho
 = x \sum_{\substack{\rho \\ \lambda\in T_1(\rho)}}\sum_{\mu\in T_1(\rho)} \psi_\mu + \lvert \lambda \rvert \sum_{\substack{\rho \\ \lambda\in T_1(\rho)}} \sum_{\gamma\in \tilde{T}_2(\rho)} \sgn(\gamma,\rho) \psi_\gamma.
 \end{gather}

 We can rewrite the first double sum of \eqref{eq:expansiontopdown} via \eqref{def:Slambda}, we get
 \begin{gather*}
 \sum_{\substack{\rho \\ \lambda \in T_1(\rho)}} \sum_{\mu \in T_1(\rho)} \psi_\mu
 =\# \{\rho \,|\, \lambda \in T_1(\rho)\} \cdot \psi_\lambda + \sum_{\mu \in S(\lambda)} \psi_\mu.
 \end{gather*}
 Next, we can use \eqref{eq:parentsandchilderen} such that we end up with
 \begin{align}
 \sum_{\substack{\rho \\ \lambda \in T_1(\rho)}} \sum_{\mu \in T_1(\rho)} \psi_\mu &=\psi_\lambda +\# T_1(\lambda) \cdot \psi_\lambda + \sum_{\mu\in S(\lambda)} \psi_\mu
=\psi_\lambda + \sum_{\gamma \in T_1(\lambda)} \sum_{\substack{\mu \\ \gamma \in T_1(\mu)}} \psi_\mu \nonumber \\
 &=\psi_\lambda +x\lvert \lambda \rvert \sum_{\gamma \in T_1(\lambda)} \psi_\gamma, \label{eq:ProofTopDown1}
 \end{align}
 where we used the induction hypothesis in the last equation.

 The second double sum of \eqref{eq:expansiontopdown} can be treated similarly. We start by rewriting the double sum such that we can apply Lemma \ref{lem:lemmafortopdown}, we get
 \begin{align*}
 \sum_{\substack{\rho \\ \lambda\in T_1(\rho)}} \sum_{\gamma \in \tilde{T}_2(\rho)} \sgn(\gamma,\rho) \psi_\gamma
 &= \sum_{\gamma \vdash \lvert \lambda \rvert -1} \bigg(\sum_{\substack{\rho \\ \lambda \in T_1(\rho) \\ \gamma \in \tilde{T}_2(\rho)}} \sgn(\gamma,\rho) \bigg) \psi_\gamma \\
 &=\sum_{\gamma \vdash \lvert \lambda \rvert -1} \bigg(\sum_{\rho \in \tilde{T}_2(\lambda) \cap T_1(\gamma)} \sgn(\rho,\lambda) \bigg) \psi_\gamma.
 \end{align*}
 If we again rewrite this equation and use the induction hypothesis, we find
 \begin{align}
 \sum_{\substack{\rho \\ \lambda\in T_1(\rho)}} \sum_{\gamma \in \tilde{T}_2(\rho)} \sgn(\gamma,\rho) \psi_\gamma
 = \sum_{\rho \in \tilde{T}_2(\lambda)} \sgn(\rho,\lambda) \sum_{\substack{\gamma \\ \rho \in T_1(\gamma)}} \psi_\gamma
 = (\lvert \lambda \rvert -1)x\sum_{\rho \in \tilde{T}_2(\lambda)} \sgn(\rho,\lambda) \psi_\rho. \!\!\!\!\label{eq:ProofTopDown2}
 \end{align}

 Bringing these results \eqref{eq:ProofTopDown1} and \eqref{eq:ProofTopDown2} together in \eqref{eq:expansiontopdown}, we obtain that
 \begin{gather*}
 \sum_{\substack{\rho \\ \lambda\in T_1(\rho)}} \psi_\rho
 = x\psi_\lambda + x\lvert \lambda \rvert \bigg(x \sum_{\gamma \in T_1(\lambda)} \psi_\gamma + (\lvert \lambda \rvert -1)\sum_{\rho \in \tilde{T}_2(\lambda)} \sgn(\rho,\lambda) \psi_\rho \bigg)
 =x (\lvert \lambda \rvert +1) \psi_\lambda,
 \end{gather*}
since we again recognize the generating recurrence relation~\eqref{eq:RecurrenceRelationHe}. This completes the proof.
\end{proof}

\section{Proof of Theorem \ref{thm:Average}: the average}\label{sectionproofaverage}

In this section we give a proof of Theorem \ref{thm:Average} which deals with the average polynomial. We prove the theorem by induction on $|\lambda|$ where we apply the generating recurrence relation \eqref{eq:RecurrenceRelationHe} in the induction step. Furthermore, we need some identities involving the numbers $F_\lambda$. The first lemma is a well-known result from representation theory \cite[formula~(1.5)]{Borodin_Olshanski}, but now also directly follows from the top-down recurrence relation.

\begin{Lemma} \label{lem:Fidentityfromabove} For any partition $\mu$, we have the identity
 \begin{gather}\label{eq:Fidentityfromabove}
 \sum_{ \substack{\lambda \\ \mu\in T_1(\lambda)}} F_{\lambda} = (|\mu| + 1) F_{\mu},
 \end{gather}
 where $T_1(\lambda)$ is as defined in Section~{\rm \ref{subsectionnotation}}.
\end{Lemma}
\begin{proof}
 Note that for every partition $\gamma$, $\He_\gamma$ is a monic polynomial. Hence, the result follows by considering the leading coefficient of both sides of~\eqref{eq:topdownrecurrenceforHe}.
\end{proof}

Another combinatorical identity that we need for the proof of the averaging result is the following. See the appendix for a proof.

\begin{Lemma}\label{lem:Fidentitytwohigher} For any partition $\lambda$, we have the identity
 \begin{gather}\label{eq:Fidentitytwohigher}
 \sum_{ \substack{\gamma \\ \lambda\in T_2^h(\gamma)}} F_{\gamma} = \sum_{ \substack{\gamma \\ \lambda\in T_2^v(\gamma)}} F_{\gamma},
 \end{gather}
 where $T_2^h(\lambda)$ and $T_2^v(\lambda)$ are as introduced in Section~{\rm \ref{subsectionnotation}}.
\end{Lemma}

The result of Lemma \ref{lem:Fidentitytwohigher} may not be new and we expect that it can also be (or already has been) proven using results in representation theory. Furthermore, the authors have also proven both identities \eqref{eq:Fidentityfromabove} and \eqref{eq:Fidentitytwohigher} by direct computation using \eqref{eq:Flambda}, but have chosen to provide only a combinatorial proof of~\eqref{eq:Fidentitytwohigher} in the appendix.

Since we have now stated the necessary lemmas, we are able to prove Theorem~\ref{thm:Average}.

\begin{proof}[Proof of Theorem \ref{thm:Average}] We prove the result by induction on $n:=|\lambda|$.

 When $n=0$, we have that $\{\lambda \vdash 0\}=\{\varnothing\}$ and $\He_{\varnothing} = 1$. Furthermore, $F_\varnothing=1$, so
 \begin{gather*}\frac{F_\varnothing}{0!} \He_\varnothing=1=x^0.\end{gather*}
 For $n=1$, we have $\{\lambda \vdash 1\}=\{(1)\}$, and $F_{(1)}=1$, and $\He_{(1)}=x$, so also for $n=1$ the result is true.

 Now suppose that $n>1$ is such that the statement is true for all $k<n$. Then, using the generating recurrence relation \eqref{eq:RecurrenceRelationHe}, we obtain that
 \begin{gather}
 \frac{1}{n!}\sum_{\lambda \vdash n} F_{\lambda}^2 \cdot \He_{\lambda}(x)
 = \frac{x}{n!} \sum_{\lambda \vdash n} \sum_{\mu\in T_1(\lambda)} F_{\lambda} F_{\mu} \He_{\mu}(x)\nonumber\\
 \hphantom{\frac{1}{n!}\sum_{\lambda \vdash n} F_{\lambda}^2 \cdot \He_{\lambda}(x)=}{}
 + \frac{|\lambda|-1}{n!} \sum_{\lambda \vdash n} \sum_{\rho\in \tilde{T}_2(\lambda)} \sgn(\rho,\lambda) F_{\lambda} F_{\rho} \He_{\rho}(x).\label{eq:ProofAverage1}
 \end{gather}
 We now consider both double sums separately. By first interchanging the sums, then applying Lemma~\ref{lem:Fidentityfromabove}, and finally using the induction hypothesis, we see that
 \begin{align}
 \sum_{\lambda \vdash n} \sum_{\mu \in T_1(\lambda)} F_\lambda F_\mu \He_\mu =\sum_{\mu \vdash n-1} \sum_{\substack{\lambda \\ \mu \in T_1(\lambda)}} F_\lambda F_\mu \He_\mu
 =\sum_{\mu \vdash n-1} (\lvert \mu \rvert +1) F_\mu^2 \He_\mu
 =n! x^{n-1}. \label{eq:ProofAverage2}
 \end{align}
 Furthermore, a similar computation, but now using Lemma \ref{lem:Fidentitytwohigher}, shows that
 \begin{align}
 \sum_{\lambda \vdash n} \sum_{\rho \in \tilde{T}_2(\lambda)} \sgn(\rho,\lambda) F_\lambda F_\rho \He_\rho &= \sum_{\rho \vdash n-2} \sum_{\substack{\lambda \\ \rho \in \tilde{T}_2(\lambda)}} \sgn(\rho,\lambda) F_\lambda F_\rho \He_\rho \nonumber \\
 &=\sum_{\rho \vdash n-2} \bigg(\sum_{\substack{\lambda \\ \rho\in T^v_2(\lambda)}} F_{\lambda} - \sum_{\substack{\lambda \\ \rho\in T^h_2(\lambda)}} F_{\lambda}\bigg) F_{\rho} \He_{\rho}(x) = 0. \label{eq:ProofAverage3}
 \end{align}
 So, using \eqref{eq:ProofAverage2} and \eqref{eq:ProofAverage3} in \eqref{eq:ProofAverage1}, we obtain that
 \begin{gather*}
 \frac{1}{n!} \sum_{\lambda \vdash n} F_\lambda^2 \He_\lambda(x) = \frac{x}{n!} n! x^{n-1}=x^n,
 \end{gather*}
 which is what we wanted to show. Hence the averaging result is established.
\end{proof}

\section{Proof of Proposition \ref{prop:OmegaDer}: the derivative}\label{sectionproofderivative}

To prove the result about the derivative of Wronskian Hermite polynomials, see Proposition \ref{prop:OmegaDer}, we again first need a combinatorial lemma.

\begin{Lemma}\label{lem:lemmaforderivative} Suppose that $n\geq 3$ is an integer and let $\lambda$ be a partition of $n$. Then, for every partition $\gamma \vdash \lvert \lambda \rvert -3$, we have that
 \begin{gather}\label{eq:ResultForDerivative}
 \sum_{\substack{\rho \in \tilde{T}_2(\lambda) \\ \gamma \in T_1(\rho)}} \sgn(\rho,\lambda)
 = \sum_{\substack{\mu \in T_1(\lambda)\\ \gamma \in \tilde{T}_2(\mu)}}\sgn(\gamma,\mu),
 \end{gather}
 where the $\sgn$ function, $T_1(\lambda)$ and $\tilde{T}_2(\lambda)$ are as defined in Section~{\rm \ref{subsectionnotation}}.
\end{Lemma}

The proof of the above lemma is postponed to the appendix. Indeed, this lemma is just a~combinatorial result and allows us to interchange sums. This is precisely how we are using it in the proof of Proposition \ref{prop:OmegaDer}; we prove the claim by induction and we need to interchange the sums in the induction step.

\begin{proof}[Proof of Proposition \ref{prop:OmegaDer}] We approach by induction on $n:=\lvert \lambda \rvert$.

 If $n \in \{0,1\}$, the result is trivial. For $n=2$, we either have $\lambda=(2)$, or $\lambda=(1,1)$. Note that in both cases $T_1(\lambda)=\{(1)\}$, that $F_{(1)}=F_{(1,1)}=F_{(2)}=1$, $\He_{(1)}(x)=x$, $\He_{(2)}(x)=x^2-1$ and $\He_{(1,1)}=x^2+1$. So indeed, $\He_{(2)}'=2\He_{(1)}$ and $\He_{(1,1)}'=2\He_{(1)}$, as desired.

 Now suppose that $n\geq 3$ and suppose that the claim holds for every partition $\mu$ such that $\lvert \mu \rvert < n$. Then, by the generating recurrence relation \eqref{eq:RecurrenceRelationHe}, we obtain
 \begin{gather*}
 \left(F_{\lambda} \He_{\lambda}\right)'(x) = \sum_{\mu\in T_1(\lambda)} F_{\mu} \He_{\mu}(x) + x \sum_{\mu\in T_1(\lambda)} \left(F_{\mu} \He_{\mu}\right)'(x) \\
 \hphantom{\left(F_{\lambda} \He_{\lambda}\right)'(x) =}{} + (|\lambda|-1) \sum_{\rho\in \tilde{T}_2(\lambda)} \sgn(\rho,\lambda) \left(F_{\rho} \He_{\rho}\right)'(x).
 \end{gather*}
 Now, applying the induction hypothesis to all the terms in the second and third sum, we arrive at
 \begin{gather*}
 \left(F_{\lambda} \He_{\lambda}\right)'(x) = \sum_{\mu\in T_1(\lambda)} F_{\mu} \He_{\mu}(x) + (|\lambda|-1) \bigg( x \sum_{\mu\in T_1(\lambda)} \sum_{\rho \in T_1(\mu)} F_{\rho} \He_{\rho}(x) \\
 \hphantom{\left(F_{\lambda} \He_{\lambda}\right)'(x) =}{} + (|\lambda|-2) \sum_{\rho\in \tilde{T}_2(\lambda)} \sgn(\rho,\lambda) \sum_{\gamma\in T_1(\rho)} F_{\gamma} \He_{\gamma}(x)\bigg).
 \end{gather*}
 After invoking Lemma \ref{lem:lemmaforderivative} for the last double sum and realizing that $\lvert\mu\rvert=\lvert\lambda\rvert-1$ for all $\mu \in T_1(\lambda)$, we see that
 \begin{gather*}
 \left(F_{\lambda} \He_{\lambda}\right)'(x) = \sum_{\mu\in T_1(\lambda)} F_{\mu} \He_{\mu}(x) + (|\lambda|-1) \sum_{\mu\in T_1(\lambda)} \bigg( x \sum_{\rho \in T_1(\mu)} F_{\rho} \He_{\rho}(x) \\
 \hphantom{\left(F_{\lambda} \He_{\lambda}\right)'(x) =}{} + (|\mu|-1) \sum_{\gamma\in \tilde{T}_2(\mu)} \sgn(\gamma,\mu) F_{\gamma} \He_{\gamma}(x)\bigg).
 \end{gather*}
 Hence, after applying the generating recurrence relation \eqref{eq:RecurrenceRelationHe} to all $\mu \in T_1(\lambda)$, we find
 \begin{gather*}(F_\lambda \He_\lambda)'=\lvert \lambda \rvert \sum_{\mu \in T_1(\lambda)} F_\mu \He_\mu,\end{gather*}
 which concludes the induction step and hence the proof.
\end{proof}

\appendix

 \section{Combinatorial lemmas}\label{appendixcombinatoriallemmas}

 \subsection{Proof of Lemma \ref{lem:lemmafortopdown}}

 \begin{proof}[Proof of Lemma \ref{lem:lemmafortopdown}] We make the distinction between $\mu\notin T_1(\lambda)$ and $\mu\in T_1(\lambda)$, and show that \eqref{eq:lemmafortopdown} holds in both cases.

 Let us first assume that $\mu\notin T_1(\lambda)$. Then suppose that the left hand side of \eqref{eq:lemmafortopdown} is non-zero. Then there exists a partition $\rho \vdash |\lambda|+1$ such that $\lambda \in T_1(\rho)$ and $\mu \in \tilde{T}_2(\rho)$. Consider the Young diagram of $\rho$ and mark the two cells that are not in the Young diagram of $\mu$ with an $A$. By definition of~$\mu \in \tilde{T}_2(\rho)$, these two cells are adjacent. Likewise, mark the unique cell in the Young diagram of $\rho$ that is not in the Young diagram of~$\lambda$ with a~$B$, see Fig.~\ref{fig:case1} for an example.
 \begin{figure}[h!]
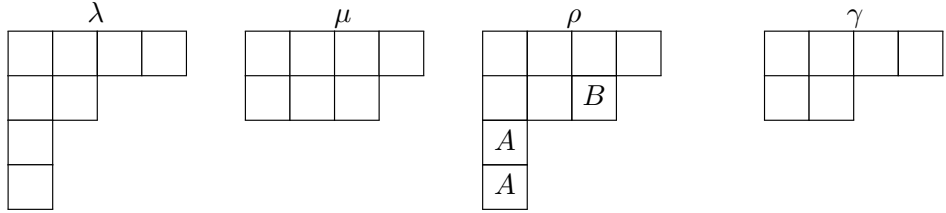

 \centering
 $\lambda$ \hspace{2.8cm} $\mu$ \hspace{2.6cm} $\rho$\hspace{3.4cm} $\gamma$ \\
 \ydiagram[*(white)]{4,2,1,1} \qquad
 \ydiagram[*(white)]{4,3} \qquad
 \ydiagram[A]{4+0,2+0,0+1,0+1}*[B]{4+0,2+1,1+0,1+0}*[*(white)]{4,2} \qquad
 \ydiagram[*(white)]{4,2}
 \caption{An illustration of the situation when $\mu\notin T_1(\lambda)$ and the left hand side of \eqref{eq:lemmafortopdown} is non-zero.}
 \label{fig:case1}
 \end{figure}

 Since $\mu\notin T_1(\lambda)$, there are now three marked cells in the Young diagram of $\rho$, i.e., the cell marked with a $B$ does not coincide with either one the $A$-cells. Furthermore, since $\lambda$ and $\mu$ are both partitions, the $B$-cell is not adjacent to either one of the $A$-cells. Therefore, removing all three marked cells in the Young diagram of $\rho$ yield a partition $\gamma\vdash |\lambda|-2$, and in fact $\gamma \in \tilde{T}_2(\lambda) \cap T_1(\mu)$, as illustrated in Fig.~\ref{fig:case1}. Furthermore, it is easy to see that $\rho$ is the only partition contributing to the left hand side of \eqref{eq:lemmafortopdown} and $\gamma$ is the only partition contributing to the right-hand side of \eqref{eq:lemmafortopdown}. Since $\sgn(\mu,\rho)=\sgn(\gamma,\lambda)$, see~\eqref{eq:sgn}, we obtain that \eqref{eq:lemmafortopdown} holds in this case. A similar argument shows that \eqref{eq:lemmafortopdown} also holds when its right hand side is non-zero and \eqref{eq:lemmafortopdown} trivially holds when both sides are zero. This establishes \eqref{eq:lemmafortopdown} if $\mu \not\in T_1(\lambda)$.

 If $\mu\in T_1(\lambda)$, there is precisely one cell in the Young diagram of $\lambda$ that is not in the Young diagram of $\mu$; mark this cell with the colour gray. We argue that \eqref{eq:lemmafortopdown} holds by making the distinction whether a cell can be added directly to the right and/or directly below the gray cell in the Young diagram of $\lambda$. Since $\mu$ is a partition, there are four cases which we treat separately.
 \begin{enumerate}\itemsep=0pt
 \item
 If a cell can be added both to the right and below the gray cell, then we are in the situation as sketched in Fig.~\ref{fig:case2}. We directly see that $\tilde{T}_2(\lambda) \cap T_1(\mu)$ is empty so the right hand side of~\eqref{eq:lemmafortopdown} is~0. Moreover, we also find that there are precisely two partitions~$\rho$ and~$\bar{\rho}$ contributing to the left hand side of~\eqref{eq:lemmafortopdown}; $\rho$~is obtained by adding a cell to the right of the gray cell and~$\bar{\rho}$ by adding a cell below the gray cell. We then have that $\sgn(\mu,\rho)=-1$ and $\sgn(\mu,\bar{\rho})=1$, so the left hand side of~\eqref{eq:lemmafortopdown} is 0 too.
 \begin{figure}[h]
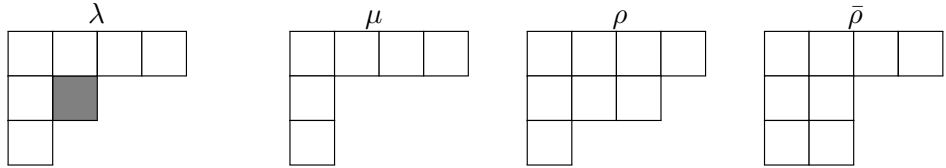

 \centering
 $\lambda$ \hspace{3.2cm} $\mu$ \hspace{2.8cm} $\rho$\hspace{2.8cm} $\bar{\rho}$ \\
 \ydiagram[*(white)]{4,1,1}*[*(gray)]{4+0,1+1,1+0} \qquad
 \ydiagram[*(white)]{4,1,1} \qquad
 \ydiagram[*(white)]{4,3,1} \qquad
 \ydiagram[*(white)]{4,2,2}
 \caption{An illustration of the situation when $\mu\in T_1(\lambda)$ and a cell can be added both to the right and below the gray cell.}
 \label{fig:case2}
 \end{figure}

 \item
 If no cell can be added to the right nor below the gray cell in the Young diagram of $\lambda$, then a similar analysis as in the first case shows that the left hand side of \eqref{eq:lemmafortopdown} is an empty sum, whereas the right hand sum equals zero. See Fig.~\ref{fig:case3} for an example.
 \begin{figure}[h]
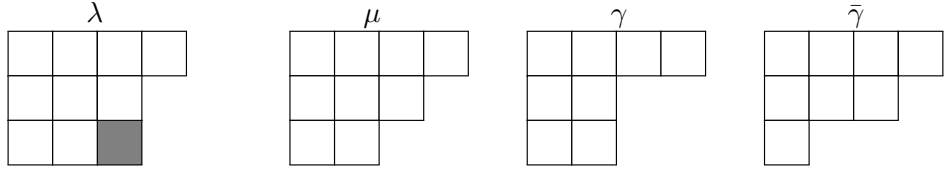

 \centering
 $\lambda$ \hspace{3.2cm} $\mu$ \hspace{2.8cm} $\gamma$\hspace{2.8cm} $\bar{\gamma}$ \\
 \ydiagram[*(white)]{4,3,2}*[*(gray)]{4+0,3+0,2+1} \qquad
 \ydiagram[*(white)]{4,3,2} \qquad
 \ydiagram[*(white)]{4,2,2} \qquad
 \ydiagram[*(white)]{4,3,1}
 \caption{An illustration of the situation when $\mu\in T_1(\lambda)$ and no cell can be added both to the right nor below the gray cell.}
 \label{fig:case3}
 \end{figure}

 \item
 If a cell can be added to the right but not below the gray cell, then both sums in \eqref{eq:lemmafortopdown} have precisely one term, and both are equal to~$-1$. See Fig.~\ref{fig:case4} for an example.
 \begin{figure}[h]
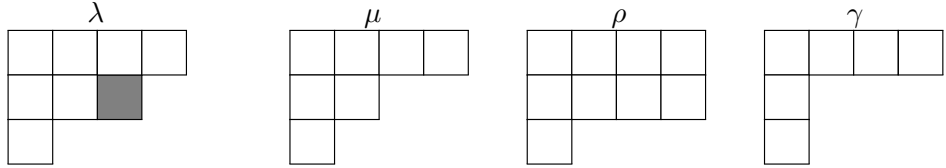

 \centering
 $\lambda$ \hspace{3.2cm} $\mu$ \hspace{2.8cm} $\rho$\hspace{2.8cm} $\gamma$ \\
 \ydiagram[*(white)]{4,2,1}*[*(gray)]{4+0,2+1,1+0} \qquad
 \ydiagram[*(white)]{4,2,1} \qquad
 \ydiagram[*(white)]{4,4,1} \qquad
 \ydiagram[*(white)]{4,1,1}
 \caption{An illustration of the situation when $\mu\in T_1(\lambda)$ and a cell can be added to the right, but not below the gray cell.}
 \label{fig:case4}
 \end{figure}

 \item
 The last case is when a cell can be added below, but not to the right of the gray cell. The same argument as in the third case can be used, but now both sides of~\eqref{eq:lemmafortopdown} are equal to~1, see Fig.~\ref{fig:case5} for an example.
 \begin{figure}[h!]
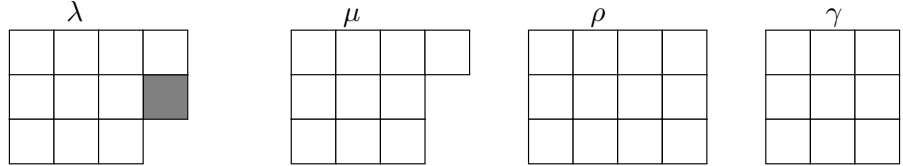

 \centering
 $\lambda$ \hspace{3.2cm} $\mu$ \hspace{2.8cm} $\rho$\hspace{2.8cm} $\gamma$ \\
 \ydiagram[*(white)]{4,3,3}*[*(gray)]{4+0,3+1,3+0} \qquad
 \ydiagram[*(white)]{4,3,3} \qquad
 \ydiagram[*(white)]{4,4,4} \qquad
 \ydiagram[*(white)]{3,3,3}
 \caption{An illustration of the situation when $\mu\in T_1(\lambda)$ and no cell can be added to the right, but there can be added a cell below the gray cell.}
 \label{fig:case5}
 \end{figure}
 \end{enumerate}
 So \eqref{eq:lemmafortopdown} also holds when $\mu\in T_1(\lambda)$. This concludes the proof.
 \end{proof}

 \subsection{Proof of Lemma \ref{lem:Fidentitytwohigher}}
 For the proof of Lemma \ref{lem:Fidentitytwohigher}, we use some new notation. For any partition $\lambda$, we define the sets $\Hor(\lambda)$ and $\Ver(\lambda)$, by
 \begin{gather*}\label{eq:defHor&Ver}
 \rho\in \Hor(\lambda) \Leftrightarrow \lambda \in T_2^h(\rho), \qquad \gamma \in \Ver(\lambda) \Leftrightarrow \lambda \in T_2^v(\gamma).
 \end{gather*}
 Indeed, in this notation, the result we prove in this section is that for any partition $\lambda$, we have
 \begin{gather}\label{eq:HorVeridentity}
 \sum_{\rho \in \Hor(\lambda)} F_\rho =\sum_{\gamma \in \Ver(\lambda)} F_\gamma.
 \end{gather}

 We prove this identity by induction on $n:=|\lambda|$. To be able to do this, we use a very basic identity.

 \begin{Lemma}\label{lem:basicFidentity}
 Suppose that $\lambda$ is a partition of $n\geq 1$. Then
 \begin{gather*}F_\lambda=\sum_{\mu\in T_1(\lambda)} F_\mu.\end{gather*}
 \end{Lemma}
 \begin{proof}
 $F_\lambda$ is the number of standard Young tableaux of shape $\lambda$. We partition the whole set of such standard Young tableaux according to the position of the number $n$. Indeed, removing the cell marked with $n$ yields a standard Young tableau of an element of $T_1(\lambda)$, and in fact this map is bijective.
 \end{proof}

 Next, for any partition $\lambda$, we study the sets
 \begin{alignat}{3}
& A_\lambda = \bigcup_{\gamma\in \Hor(\lambda)} T_1(\gamma), \qquad && B_\lambda = \bigcup_{\mu \in T_1(\lambda)} \Hor(\mu),& \nonumber\\
& C_\lambda = \bigcup_{\gamma\in \Ver(\lambda)} T_1(\gamma), \qquad && D_\lambda = \bigcup_{\mu \in T_1(\lambda)} \Ver(\mu).& \label{eq:ABCD}
 \end{alignat}
 In particular, $A_\lambda$, $B_\lambda$, $C_\lambda$ and $D_\lambda$ are all sets of partitions of $\lvert \lambda \rvert +1$. We need the following properties of these sets in the proof of~\eqref{eq:HorVeridentity}.

 \begin{Lemma}\label{lem:propertiesABCD}
 Suppose that $\lambda$ is a partition. Then the following three statements hold:
 \begin{enumerate}\itemsep=0pt
 \item[$1.$] the defining expressions for $A_\lambda$, $B_\lambda$, $C_\lambda$ and $D_\lambda$ in \eqref{eq:ABCD} are disjoint unions,
 \item[$2.$] $A_\lambda \setminus B_\lambda = C_\lambda \setminus D_\lambda$,
 \item[$3.$] $B_\lambda \setminus A_\lambda = D_\lambda \setminus C_\lambda$.
 \end{enumerate}
 \end{Lemma}
 \begin{proof}
 The first item is straightforward from definitions.

 For the second item, note that if $\rho$ is an element of $A_\lambda$ and $\lambda \not\in T_1(\rho)$, then $\rho\in B_\lambda$ too. Therefore, if $\rho \in A_\lambda \setminus B_\lambda$, we certainly have that $\lambda \in T_1(\rho)$. Now consider the Young diagram of $\rho$ and colour the (unique) cell that is not in $\lambda$. Since $\rho \in A_\lambda$ but $\rho \not\in B_\lambda$, we are able to add a cell to the right of the coloured cell and below the coloured cell, like in the example indicated in Fig.~\ref{fig:AwithoutB}. With the same reasoning it is easy to see that the elements of $C_\lambda \setminus D_\lambda$ are also given by this same condition, hence the second item holds.

\begin{figure}[t]
 \centering
 \ydiagram[*(gray)]{4+0,2+1,2+0,1+0}*[*(white)]{4,2,2,1}
 \caption{An example that indicates an element $\rho \in A_\lambda \setminus B_\lambda$.} \label{fig:AwithoutB}
 \vspace{0.25cm}
 \ydiagram[*(gray)]{5+0,3+0,2+1,1+0}*[*(white)]{5,3,2,1}
 \caption{An example that indicates an element $\rho \in B_\lambda \setminus A_\lambda$. } \label{fig:BwithoutA}
 \end{figure}

 For the third item, note that the elements of $B_\lambda\setminus A_\lambda$ have to be partitions $\rho$ such that $\lambda \in T_1(\rho)$. Again, we colour the unique cell in the Young diagram of $\rho$ that is not in the Young diagram of $\lambda$. Since $\rho \in B_\lambda$ and $\rho \not\in A_\lambda$, we must then have that no cell can be added to the right of the coloured cell and neither a cell below the coloured cell, like in the example in Fig.~\ref{fig:BwithoutA}. It is easy to check that the elements of $D_\lambda \setminus C_\lambda$ have the same property; hence this establishes the third item.
 \end{proof}

We can now turn to the proof of Lemma \ref{lem:Fidentitytwohigher}. Recall that we are actually proving the equivalent statement \eqref{eq:HorVeridentity}. For simplicity, we write $A$ for $A_\lambda$, $B$ for $B_\lambda$, and so on.

 \begin{proof}[Proof of Lemma \ref{lem:Fidentitytwohigher}] We prove this by induction on $n:=\lvert \lambda \rvert$.

 When $n=0$, then $\lambda=\varnothing$, $\Hor(\lambda)=\{(2)\}$ and $\Ver(\lambda)=\{(1,1)\}$, and $F_{(2)}=F_{(1,1)}=1$, which establishes the required identity.

 Now suppose that we have proven the result for all partitions $\mu$ such that $\lvert \mu \rvert <n$. Then in particular, for all $\mu\in T_1(\lambda)$, we know that
 \begin{gather*}\sum_{\gamma \in \Hor(\mu)} F_\gamma= \sum_{\gamma \in \Ver(\mu)} F_\gamma.\end{gather*}
 Using this, we can now prove the induction step. By first using Lemma \ref{lem:basicFidentity}, then the first item of Lemma \ref{lem:propertiesABCD}, some basic set theory and again the first item of Lemma~\ref{lem:propertiesABCD}, we see that
 \begin{align}
 \sum_{\gamma \in \Hor(\lambda)} F_\gamma &= \sum_{\gamma \in \Hor(\lambda)} \sum_{\rho \in T_1(\gamma)} F_\rho = \sum_{\rho \in A} F_\rho
 =\sum_{\rho \in A\setminus B} F_\rho -\sum_{\rho \in B\setminus A} F_\rho + \sum_{\rho \in B} F_\rho \nonumber\\
 &=\sum_{\rho \in A\setminus B} F_\rho -\sum_{\rho \in B\setminus A} F_\rho + \sum_{\mu \in T_1(\lambda)} \sum_{\rho \in \Hor(\mu)} F_\rho. \label{eq:Horrewritten}
 \end{align}
 In a similar fashion, it follows that
 \begin{gather}\label{eq:Verrewritten}
 \sum_{\gamma \in \Ver(\lambda)} F_\gamma =\sum_{\rho \in C\setminus D} F_\rho -\sum_{\rho \in D\setminus C} F_\rho + \sum_{\mu \in T_1(\lambda)} \sum_{\rho \in \Ver(\mu)} F_\rho.
 \end{gather}

 We see that \eqref{eq:Horrewritten} and \eqref{eq:Verrewritten} coincide, by using Lemma \ref{lem:propertiesABCD} for the first two sums and the induction hypothesis for the last sum. Hence we have established \eqref{eq:HorVeridentity}.
 \end{proof}

 \subsection{Proof of Lemma \ref{lem:lemmaforderivative}}

 \begin{proof}[Proof of Lemma \ref{lem:lemmaforderivative}]
 Fix a partition $\lambda$. Let $\gamma\vdash |\lambda|-3$. Take $k=3$ in definition \eqref{eq:Tk}. Then, $\gamma\notin T_3(\lambda)$ or $\gamma\in T_3(\lambda)$.

 Clearly, if $\gamma\notin T_3(\lambda)$, then both sides of \eqref{eq:ResultForDerivative} are empty sums. If $\gamma\in T_3(\lambda)$, write the left hand side and right hand side of \eqref{eq:ResultForDerivative} as
 \begin{gather*}
 c_{\gamma}= \sum_{ \substack{\rho\in \tilde{T}_2(\lambda) \\ \gamma\in T_1(\gamma)}} \sgn(\rho,\lambda), \qquad
 d_{\gamma}= \sum_{ \substack{\mu\in T_1(\lambda) \\ \gamma\in \tilde{T}_2(\mu)}} \sgn(\rho,\lambda).
 \end{gather*}
 By definition \eqref{eq:sgn}, both values $c_{\gamma}$ and $d_{\gamma}$ are integers. Our goal is now to prove that $c_\gamma=d_\gamma$, since then~\eqref{eq:ResultForDerivative} is established. We split up in several cases. For this, suppose that $\gamma \in T_3(\lambda)$ and consider the Young diagrams of $\lambda$ and $\gamma$; there are three cells of the Young diagram of $\lambda$ that are not in the Young diagram of $\gamma$ and the rest of the cells is precisely the Young diagram of $\gamma$. We argue by how these three cells are positioned with respect to each other.

 \begin{enumerate}\itemsep=0pt
 \item If all three are pairwise non-adjacent (either vertically or horizontally), then both sums are empty and hence we have $c_\gamma=d_\gamma=0$.

 \item If two of the cells are adjacent and the third one is not adjacent to either one of these two cells, then we label the adjacent pair by $A$ and $B$, and the third loose cell with $C$. If $A$ and $B$ are horizontally adjacent, then $c_\gamma=d_\gamma=-1$. This can be seen by observing that the only $\rho\in \tilde{T}_2(\lambda)$ such that $\gamma \in T_1(\rho)$ is precisely the one by removing cells $A$ and $B$ from $\lambda$. Likewise, the only $\mu\in T_1(\lambda)$ such that $\gamma \in \tilde{T}_2(\mu)$ is precisely the one by removing cell $C$ from $\lambda$. Similarly, if $A$ and $B$ are vertically adjacent, then $c_\gamma=d_\gamma=1$.

 \item The most elaborate case is the case where the three cells form an adjacent group of cells in the Young diagram of $\lambda$. We then have that there is one cell (labeled by $B$) that is adjacent to both the other cells (labeled by $A$ and $C$) and $A$ and $C$ are not adjacent. Since both $\lambda$ and $\gamma$ are partitions, there are now four cases (up to interchanging the labels $A$ and $C$) for the alignment of $A$, $B$ and $C$ in the Young diagram, that we treat case by case.

{\bf Alignment 1} \ \begin{ytableau} A & B & C \end{ytableau} \vspace{0.15cm}

 There is precisely one $\rho \in \tilde{T}_2(\lambda)$ such that $\gamma\in T_1(\rho)$; indeed, the one obtained by remo\-ving~$B$ and $C$ together from the Young diagram of~$\lambda$. Therefore $c_\gamma=-1$. Likewise, $d_\gamma=-1$.

{\bf Alignment 2} \ \begin{ytableau} \none & C \\ A & B \end{ytableau} \vspace{0.15cm}

 There are now two elements $\rho \in \tilde{T}_2(\lambda)$ such that $\gamma \in T_1(\rho)$; indeed, the one obtained by removing $A$ and $B$
 and the one obtained by removing $B$ and $C$. Note that we therefore have that $c_\gamma=1-1=0$. Furthermore, there is no $\mu \in T_1(\lambda)$ such that $\gamma \in \tilde{T}_2(\mu)$, so $d_\gamma=0$ too.

{\bf Alignment 3} \ \begin{ytableau} B & C \\ A \end{ytableau} \vspace{0.15cm}

 This is similar to the second alignment case. We now have that $c_\gamma=0$ and $d_\gamma=1-1=0$ too.

{\bf Alignment 4} \ \begin{ytableau} A \\ B \\ C \end{ytableau} \vspace{0.15cm}

 This is very similar to the first alignment case; we now have that $c_\gamma=d_\gamma=1$.
 \end{enumerate}

 So, we see that in all cases we have that $c_\gamma=d_\gamma$, and that concludes the proof.
 \end{proof}

\subsection*{Acknowledgements}
We thank Arno Kuijlaars and Walter Van Assche for fruitful discussions and a careful reading of a preliminary version of this article. The authors are supported in part by the long term structural funding-Methusalem grant of the Flemish Government, and by EOS project 30889451 of the Flemish Science Foundation (FWO). Marco Stevens is also supported by the Belgian Interuniversity Attraction Pole P07/18, and by FWO research grant G.0864.16.

\pdfbookmark[1]{References}{ref}
\LastPageEnding

\end{document}